\documentclass[11pt]{amsart}
\newcommand{\deleted}[1]{}
\newcommand{\delete}[1]{}
\newcommand{\mynotes}[1]{}
\newcommand\notes[1]{}
\usepackage{txfonts}
\usepackage{amssymb}
\usepackage{eucal}
\usepackage{graphicx}
\usepackage{amsmath}
\usepackage{amscd}
\usepackage[all]{xy}
\usepackage[bookmarksnumbered=true,
            bookmarksopen=true,
            colorlinks,
            pdfborder=001,
            linkcolor=blue,
            anchorcolor=blue,
            citecolor=blue
            ]{hyperref}

\usepackage{amsthm}
\usepackage{mathrsfs}
\usepackage{enumitem}
\usepackage{amsfonts,latexsym}
\usepackage{xspace}
\usepackage{epsfig}
\usepackage{float}
\usepackage{color}
\usepackage{fancybox}
\usepackage{colordvi}
\usepackage{multicol}
\usepackage{tikz}
\usepackage{wasysym}
\usepackage{xspace}
\usepackage[active]{srcltx} 
\usepackage{mathtools} 
\usepackage{tabularx}
\usepackage{stfloats}
\usepackage{makecell}
\usepackage{caption}
\newcommand\changed[1]{#1}
\changed{}

\newtheorem{theorem}{Theorem}[section]
\newtheorem{lemma}[theorem]{Lemma}

\newtheorem{question}[theorem]{Question}

\newtheorem{prop}[theorem]{Proposition}

\theoremstyle{definition}
\newtheorem{defn}[theorem]{Definition}

\newtheorem{exam}[theorem]{Example}

\deleted{\newtheorem{theorem}{Theorem}[section]
\newtheorem{prop-def}{Proposition-Definition}[section]
\newtheorem{coro-def}{Corollary-Definition}[section]

}
\newcommand{\nc}{\newcommand}

\newcommand{\efootnote}[1]{}

\renewcommand\geq{\geqslant}
\renewcommand\leq{\leqslant}

\newcommand{\lra}{\longrightarrow}

\newcommand{\lraf}[1]{\stackrel{#1}{\lra}}

\newcommand{\raf}[1]{\stackrel{#1}{\ra}}

\newcommand{\ra}{\rightarrow}
\newcommand{\lan}{\langle}
\newcommand{\ran}{\rangle}

\newcommand{\cpx}[1]{#1^{\bullet}}

\newcommand{\Db}[1]{{\mathscr D}^b(#1)}
\newcommand{\Kb}[1]{{\mathscr K}^b(#1\mbox{{\rm -proj}})}
\newcommand{\Ka}{{\mathscr K}^{-, b}(A\mbox{{\rm -proj}})}
\newcommand{\KL}{{\mathscr K}^{-, b}(\Lambda\mbox{{\rm -proj}})}
\newcommand{\Kbb}[1]{{\mathscr K}^{-, b}(\C)}
\newcommand{\K}{{\mathscr K}}
\newcommand{\Kc}{{\mathscr K}(\C)}
\newcommand{\Cc}{{\mathscr C}(\C)}

\newcommand{\md}{\ensuremath{\mbox{{\rm -mod}}}}

\nc{\mrm}[1]{{\rm #1}}

\DeclareMathOperator{\Zmap}{\mathcal{Z}}

\nc{\Hom}{\mrm{Hom}}
\nc{\Ext}{\mrm{Ext}}
\nc{\End}{\mrm{End}}
\nc{\rad}{\mrm{rad}}
\nc{\Aut}{\mrm{Aut}}
\nc{\A}{\mathcal{A}}
\nc{\B}{\mathcal{B}}
\nc{\C}{\mathcal{C}}
\nc{\M}{\mathcal{M}}
\nc{\X}{X^{\bullet}}
\nc{\Y}{Y^{\bullet}}
\nc{\Z}{\mathbb{Z}}
\nc{\T}{T^{\bullet}}
\nc{\U}{U^{\bullet}}
\nc{\V}{V^{\bullet}}
\nc{\f}{f^{\bullet}}
\nc{\g}{g^{\bullet}}
\nc{\h}{h^{\bullet}}
\nc{\rr}{r^{\bullet}}
\nc{\s}{s^{\bullet}}
\nc{\Lv}{{\bf L}\nu}
\nc{\Lm}{{\bf L}\mu}
\nc{\Lx}{{\bf L}\xi}
\nc{\Pv}{{\bf P}\nu}
\nc{\w}{\widetilde}
\nc{\N}{\mathcal{N}}
\nc{\cw}[1]{\widetilde{#1}^{\bullet}}
\nc{\tr}{\mrm{tr}}
\nc{\pr}{\mrm{Pic}}
\nc{\Out}{\mrm{Out}}
\nc{\ch}{\mrm{char}}
\nc{\add}{\mrm{add}}
\nc{\RHom}{\mrm{{\bf R}Hom}}
\nc{\OT}{\otimes^{\bf L}}
\nc{\thick}{\mrm{thick}}
\nc{\cone}{\mrm{cone}}
\nc{\proj}{\mbox{{\rm -proj}}}
\nc{\Proj}{\mrm{Proj\,}}
\nc{\gl}{\mrm{gl}}
\nc{\im}{\mrm{im}\,}
\nc{\id}{\mrm{id}}
\nc{\ind}{\mrm{ind}}
\nc{\lcd}[1]{\mrm{lcd}(#1)}
\nc{\diag}{\mrm{diag}}
\nc{\pid}{-pseudo-identity\;}
\nc{\tp}[1]{\mrm{TrPic}(#1)}
\nc{\res}{\mrm{res}}
\nc{\PI}[1]{\mrm{PI}(#1)}
\nc{\uu}[1]{T(#1)}
\nc{\soc}{\mrm{soc}}
\nc{\va}{\varepsilon}

\begin{document}
\renewcommand{\thefootnote}{\alph{footnote}}
\setcounter{footnote}{-1}
\footnote{2020 Mathematics Subject Classification: 18G80, 13D09.}
\renewcommand{\thefootnote}{\alph{footnote}}
\setcounter{footnote}{-1}
\footnote{Keywords: Derived category, Derived equivalence, Standard derived equivalence, Triangle functor, Truncated triangle.}

\title[Rickard's question on derived equivalences]{Rickard's question on standard derived equivalences}
\author[Wei Hu]{Wei Hu}
\address{Wei Hu: School of Mathematical Sciences, Beijing Normal University, 100875 Beijing, China}
\email{huwei@bnu.edu.cn}
\author[Changchang Xi]{Changchang Xi}
\address{Changchang Xi: School of Mathematical Sciences, Capital Normal University, 100048 Beijing, China}
\email{xicc@cnu.edu.cn}
\author[Jin Zhang]{Jin Zhang}
\address{Jin Zhang: School of Mathematics and Statistics, Lanzhou University, 730000 Lanzhou, \&
School of Mathematical Sciences, Laboratory of Mathematics and Complex Systems, MOE, Beijing Normal University, 100875 Beijing, China}
\email{zj\_10@lzu.edu.cn}

\begin{abstract}
In 1991 Rickard asked whether every derived equivalence of algebras over a common field is standard. We construct a series of examples of non-standard derived equivalences, and thus answer the question negatively. Moreover, we give a necessary and sufficient condition for the question to be true.
\end{abstract}

\maketitle

\tableofcontents

\section{Introduction}
Let $A$ and $B$ be finite-dimensional algebras over a field. Denote by $\Db{A}$ the bounded derived category of the category $A\md$ of finitely generated $A$-modules. An equivalence $F: \Db{A}\ra \Db{B}$ is called a {\it derived equivalence} if $F$ is a triangle equivalence \cite{H, R0}. Moreover, if $F$, in addition, is naturally isomorphic to the derived tensor functor induced by a bounded complex of $B$-$A$-bimodules, then $F$ is called a {\it standard} derived equivalence \cite{R}. Clearly, standard derived equivalences form an important class of derived equivalences and have many significant applications, such as lifting or restricting to stable equivalences of Morita type, and calculating homological invariants and groups (see \cite{R, RZ, Y}).

How can an arbitrary derived equivalence be related to a standard derived equivalence? Rickard showed in \cite[Corollary 3.5]{R} that, given  a derived equivalence $F: \Db{A}\ra \Db{B}$, there always exists a standard derived equivalence $\w{F}: \Db{A}\ra \Db{B}$ such that $F$ and $\w{F}$ agree on the full subcategory of $\Db{A}$ consisting of projective $A$-modules and $\w{F}(\X)\simeq F(\X) $ for every object $\X$ of $\Db{A}$. However, it remains  unclear whether $F$ and $\w{F}$ are naturally isomorphic. Hence he proposed the following fundamental question \cite{R}.
\begin{question}{\rm \cite{R} }
 Are all derived equivalences standard?
\end{question}

We say that Rickard's question \emph{holds for} a finite-dimensional algebra $A$ over a field if every derived auto-equivalence of $\Db{A}$ is standard. It is known that Rickard's question holds for $A$ if and only if every derived equivalence $\Db{A}\ra \Db{B}$ between $A$ and an arbitrary algebra $B$ is standard.

Neeman \cite{N} investigated the question with the emphasis on morphisms between triangles and obtained a positive answer by adding axioms to the notion of triangulated categories. However, the question remains open in general. Only a few cases have been verified. For instance, a positive answer has been obtained for hereditary algebras \cite{MY}, triangular algebras \cite{C}, and derived-discrete algebras \cite{BC, CZ}. All of these results seem to suggest a positive answer to Rickard's question.

In the present paper, we will provide a negative answer to Rickard's question. This is done by understanding derived equivalences of special form. During this course, we are led to construct a series of counterexamples.

Following \cite{CY}, a derived auto-equivalence $(F, \xi): \Db{A}\ra \Db{A}$  (respectively,  $\Kb{A}\ra \Kb{A}$) is called a {\it pseudo-identity} if $F$ fixes all objects of $\Db{A}$ (respectively, $\Kb{A}$) and fixes $A\md[i]$ (respectively, $A\proj[i]$) for all $i\in \Z$, where $[i]$ is the $i$-th shift functor. Here, $A$-proj is the category of finitely generated projective $A$-modules, and $\Kb{A}$ denotes the bounded homotopy category of $A$-proj.

First, we give a criterion for Rickard's question to hold.

\begin{theorem}\label{t0}
Let $A$ be a finite-dimensional algebra over a field. Then Rickard's question holds for $A$ if and only if every pseudo-identity $(F, \id): \Kb{A}\ra \Kb{A}$ that fixes all truncated triangles is naturally isomorphic to the identity triangle functor $(\id,\id)$.
\end{theorem}

For the definition of truncated triangles, we refer to Definition \ref{def-truncated}. It was proved in \cite[Theorem 6.1]{CY} that Rickard's question holds for $A$ if every pseudo-identity $(F, \xi): \Kb{A}\ra \Kb{A}$ is naturally isomorphic to the identity  $(\id,\id)$. By Theorem \ref{t0}, we do not have to check every pseudo-identity. It is enough to verify only those pseudo-identities  $(F, \id)$ that fix all truncated triangles. As an application of Theorem \ref{t0}, we give a short proof of Rickard's question for derived-discrete algebras in some cases; see Example \ref{ex}.

Based on Theorem \ref{t0} and a reduction lemma (see Lemma \ref{coro}), we have the following negative answer to Rickard's question. The counterexamples are given in Example \ref{counterexa}.

\begin{theorem}\label{t}
There exist infinitely many non-standard derived equivalences between finite-dimensional algebras over a common field.
\end{theorem}

The paper is outlined as follows: In Section \ref{pre}, we give basic definitions and prove some preparatory results. In Sections \ref{aut} and \ref{example}, we prove Theorems \ref{t0} and \ref{t}, respectively. Also, we display explicit counterexamples in Section \ref{example}.

\section{Preliminaries}\label{pre}
In this section, we introduce basic notation and establish preparatory results for proofs of the main results in Introduction.

Let $R$ be a ring with identity. Denote by $R\md$ the category of finitely generated left $R$-modules and $R\proj$ the full subcategory of $R\md$ consisting of projective $R$-modules.

Let $\C$ be an additive category. Denote by $\Cc$ the category of complexes over $\mathcal{C}$, and by [1] the shift functor. Let $\Kc$ be the homotopy category of $\C$, and let $\K^b(\C)$ be the bounded homotopy category of $\C$. If $\C$ is a full subcategory of an abelian category $\A$, we denote by $\Kbb{\C}$ the homotopy category of bounded-above complexes over $\mathcal{C}$ with bounded homology. Given complexes $\X=(X^i,d^i_X: X^i\ra X^{i+1}), \Y=(Y^i, d_Y^i:Y^i\ra Y^{i+1})$ in $\Cc$, and two morphisms $\f=(f^i), \g=(g^i):\X \to \Y$ in $\Cc$, we write $\f\sim \g$ if they are homotopic, that is, there are maps $c^i: X^i\ra Y^{i-1}$ such that $f^i-g^i=d_{Y}^{i-1}c^i+c^{i+1}d_X^i$ for all $i \in \mathbb{Z}$. In this case, $\cpx{c}=(c^i): \X\ra \Y[-1]$
 is called a \emph{homotopy} for $\f\sim \g$.

For simplicity, we use the same symbol $\f$ for its homotopy class in $\Kc$.
For other notation, definitions and basic facts about triangulated categories, we refer to \cite{H}.

A morphism $f: X\ra Y$ in an additive category $\C$ is said to be {\it radical} if $\id_Y-fg$ is an isomorphism for any $g\in \Hom_{\C}(Y, X)$.  If $\C$ is a Krull-Schmidt category, then this is equivalent to saying that for every
indecomposable object $Z$ and every pair of morphisms
$h:Z\to X$ and $g:Y\to Z$, the morphism $gfh$ is not an isomorphism (see, for example, \cite[Corollary 2.10]{Krause}).

Assume that $\C$ is a Krull--Schmidt category. Then every morphism in $\C$ decomposes into an isomorphism and a radical morphism. More precisely, given a morphism $f: X\ra Y$ in $\C$, there are decompositions $X=X_1\oplus X_2$ and $Y=Y_1\oplus Y_2$ and automorphisms $s_X: X\ra X$ and $s_Y: Y\ra Y$ such that
$$s_Yfs_X^{-1}=\left(\begin{matrix}
s& 0\\
0&r
\end{matrix}\right): X_1\oplus X_2\lra Y_1\oplus Y_2$$
 where $s: X_1\ra Y_1$ is an isomorphism and $r:X_2\ra Y_2$ is a radical morphism.

\subsection{Morphisms between standard triangles}\label{triangles}
Let $\C$ be an additive category. Let $\X=(X^i, d_X^i)$ and $\Y=(Y^i, d_Y^i)$ be complexes over $\C$. By a \emph{graded} morphism $\g: \X\ra \Y$ of complexes over $\C$, we mean a family of morphisms $g^i: X^i\ra Y^i$ in $\C$ indexed by $i\in \Z$. In general, a graded morphism may not be a morphism of complexes.

Suppose that $\f:\X\to\Y$ is a morphism of complexes. The \emph{mapping cone} of $\f$ is defined by
$$\cone(\f):=(X^{i+1}\oplus Y^i, \left(\begin{matrix}
-d_X^{i+1}& 0\\
f^{i+1}&d_Y^i
\end{matrix}\right): X^{i+1}\oplus Y^i\ra X^{i+2}\oplus Y^{i+1}).$$
The sequence
$$\Delta(\f): \X\stackrel{\f}\lra \Y\stackrel{\theta(\f)}\lra \cone(\f)\stackrel{\pi(\f)}\lra \X[1]$$
in $\Kc$, with
$ \theta(\f):=\left(\begin{smallmatrix}
0\\
1
\end{smallmatrix}\right) \text{\;and\;} \pi(\f):=(1, 0),$
is called a \emph{standard triangle}. Every triangle in $\Kc$ is isomorphic to a standard triangle. In this way, $\Kc$ is a triangulated category.

Given morphisms $\f:\X\to \Y$, $\g: \U\to \V$, $\cpx{r}: \Y\to \V$ and $\cpx{s}: \X\to \U$ of complexes over $\C$, we consider the diagram
 \[(\divideontimes)\;\;\;\;\xymatrix@R=1.5em{
\X\ar[r]^{\f}&\Y\ar[r]^-{\theta(\f)}\ar[d]^{\cpx{r}}& \cone(\f)\ar[r]^{\pi(\f)}
\ar@{..>}[d]&\X[1]\ar[d]^{\cpx{s}[1]} \\
\U\ar[r]^{\g}&\V\ar[r]^-{\theta(\g)}& \cone(\g)\ar[r]^{\pi(\g)}&\U[1]
}\]
in $\Kc$ with rows being standard triangles. In what follows, we write a graded morphism from $\cone(\f)$  to $\cone(\g)$ as $\left(\begin{smallmatrix}
\cpx{a}& \cpx{b}\\
\cpx{c}& \cpx{d}
\end{smallmatrix}\right)$, where its $i$-component is $\left(\begin{smallmatrix}
a^i& b^i\\
c^i& d^i
\end{smallmatrix}\right): X^{i+1}\oplus Y^i\ra U^{i+1}\oplus V^i$ with $a^i: X^{i+1}\to U^{i+1}, b^i:Y^i\to U^{i+1}, c^i:X^{i+1}\to V^i$ and $ d^i:Y^i\to V^i$ being morphisms in $\C$.

The following lemma demonstrates the existence and possible choices of a morphism from $\cone(\f)$ to $\cone(\g)$ in $(\divideontimes)$.

\begin{lemma}\label{st}
Keep the notation above. Then the following statements hold.

$(1)$ If $\left(\begin{smallmatrix}
\cpx{a}& \cpx{b}\\
\cpx{c}& \cpx{d}
\end{smallmatrix}\right): \cone(\f)\ra \cone(\g)$ is a morphism of complexes, making one of the squares in $(\divideontimes)$ commute in $\Kc$, then it is homotopic to a morphism of the form $\left(\begin{smallmatrix}
\cpx{x}&0 \\
\cpx{y}&\cpx{z}
\end{smallmatrix}\right)$.

$(2)$  A graded morphism $\left(\begin{smallmatrix}
\cpx{a}& 0\\
\cpx{c}& \cpx{d}
\end{smallmatrix}\right): \cone(\f)\ra \cone(\g)$
is a morphism of complexes if and only if $\cpx{a}: \X[1]\ra \U[1]$ and $\cpx{d}: \Y\ra \V$ are morphisms of complexes such that $\cpx{d} \f\sim \g(\cpx{a}[-1])$ with $\cpx{c}[-1]$ as a homotopy.

$(3)$ There is a morphism $\cpx{\phi}:\cone(\f)\ra \cone(\g)$ of complexes, making both squares in $(\divideontimes)$ commute if and only if there is a morphism $\h: \Y\ra \U$ of complexes such that $\g \h \f\sim \g\cpx{s}-\cpx{r}\f$. In this case, $\cpx{\phi}\sim\left(\begin{smallmatrix}
\cpx{s}[1]&0 \\
\cpx{c}[1]&\cpx{r}+\g\h
\end{smallmatrix}\right): \cone(\f)\ra \cone(\g)$ with $\cpx{c}$ a homotopy for $(\cpx{r}+\g\h)\f\sim \g\cpx{s}$.

$(4)$  If two morphisms $\cpx{\phi}, \cpx{\psi}: \cone(\f)\ra \cone(\g)$ of complexes make the squares in $(\divideontimes)$ commute, then there exists a morphism $\h: \Y\ra \U$ of complexes such that $\g\h\f\sim 0$  and
$\cpx{\phi}-\cpx{\psi}\sim\left(\begin{smallmatrix}
0&0 \\
\cpx{c}[1]&\g\h
\end{smallmatrix}\right)$, where $\cpx{c}$ is a homotopy for $\g\h\f\sim 0$.

$(5)$ Assume that $\cpx{r}$ and $\cpx{s}$ in $(\divideontimes)$ are isomorphisms and $\g$ is a radical morphism in $\Kc$. If there exists a morphism $\cpx{\phi}:\cone(\f)\ra \cone(\g)$  of complexes making both squares in $(\divideontimes)$ commute in $\Kc$, then $\cpx{\phi}$ is an isomorphism in $\Kc$.
\end{lemma}

\begin{proof}
(1) Assume that the left square in $(\divideontimes)$ commutes. Then
$$\left(\begin{matrix}
\cpx{b}\\
\cpx{d}
\end{matrix}\right)= \left(\begin{matrix}
\cpx{a}& \cpx{b}\\
\cpx{c}& \cpx{d}
\end{matrix}\right)\circ \theta(\f)\sim \theta(\g)\circ \cpx{r}=\left(\begin{matrix}
0\\
\cpx{r}
\end{matrix}\right).$$
Let $\left(\begin{smallmatrix}
\cpx{m}\\
\cpx{n}
\end{smallmatrix}\right): \Y\ra \cone(\g)[-1]$ be a homotopy for $\left(\begin{smallmatrix}
\cpx{b}\\
\cpx{d}
\end{smallmatrix}\right)\sim\left(\begin{smallmatrix}
0\\
\cpx{r}
\end{smallmatrix}\right)$, that is,
\[\left(\begin{matrix}
-d_U^i&0\\
g^i&d_V^{i-1}
\end{matrix}\right)\left(\begin{matrix}
m^i\\
n^i
\end{matrix}\right)+\left(\begin{matrix}
m^{i+1}\\
n^{i+1}
\end{matrix}\right)d_Y^i=\left(\begin{matrix}
b^i\\
d^i-r^i
\end{matrix}\right)\]
 for all $i$. Then $-d_U^im^i+m^{i+1}d_Y^i=b^i$. This gives a graded morphism $\left(\begin{smallmatrix}
0 &-\cpx{m}\\
0&0
\end{smallmatrix}\right): \cone(\f)\ra \cone(\g)[-1]$ such that
$$\left(\begin{matrix}
0&-m^{i+1}\\
0&0
\end{matrix}\right)\left(\begin{matrix}
-d_X^{i+1}&0\\
f^{i+1}&d_Y^i
\end{matrix}\right)+\left(\begin{matrix}
-d_U^i&0\\
g^i&d_V^{i-1}
\end{matrix}\right)\left(\begin{matrix}
0&-m^i\\
0&0
\end{matrix}\right)+\left(\begin{matrix}
a^i&b^i\\
c^i&d^i
\end{matrix}\right)=\left(\begin{matrix}
a^i-m^{i+1}f^{i+1}&0\\
c^i&d^i-g^i m^i
\end{matrix}\right).$$
This means
$$\left(\begin{matrix}
\cpx{a}& \cpx{b}\\
\cpx{c}& \cpx{d}
\end{matrix}\right)\sim\left(\begin{matrix}
\cpx{a}-(\cpx{m}[1])(\f[1])&0\\
\cpx{c}&\cpx{d}-\g\cpx{m}
\end{matrix}\right): \cone(\f)\lra \cone(\g).$$
Similarly, we can show the case that the right square commutes.

(2) This follows by definition.

(3)  Assume that there is a morphism $\left(\begin{smallmatrix}
\cpx{a}& \cpx{b}\\
\cpx{c}& \cpx{d}
\end{smallmatrix}\right): \cone(\f)\ra \cone(\g)$ of complexes, making the squares in $(\divideontimes)$ commute. By (1), we may assume $\cpx{b}=0$. The commutativity of the left square implies that
$$\left(\begin{matrix}
0\\
\cpx{d}
\end{matrix}\right)= \left(\begin{matrix}
\cpx{a}& 0\\
\cpx{c}& \cpx{d}
\end{matrix}\right)\circ \theta(\f)\sim \theta(\g)\circ \cpx{r}=\left(\begin{matrix}
0\\
\cpx{r}
\end{matrix}\right).$$
Thus $\theta(\g) (\cpx{d}-\cpx{r})=\left(\begin{smallmatrix}
0\\
\cpx{d}-\cpx{r}
\end{smallmatrix}\right)\sim 0$. Hence there is a morphism $\h: \Y\ra \U$ of complexes such that $\cpx{d}-\cpx{r}\sim \g\h$. Similarly, by the commutativity of the right square, there is a morphism $\cpx{m}: \Y[1]\ra \U[1]$ of complexes such that $\cpx{a}-\cpx{s}[1]\sim \cpx{m}(\f[1])$. Let $\cpx{w}: \X[1]\ra (\U[1])[-1]$ be a homotopy for $\cpx{a}-\cpx{s}[1]\sim\cpx{m}(\f[1])$, that is, $w^{i+1}(-d_X^{i+1})+(-d_U^i)w^i=a^i-s^{i+1}-m^if^{i+1}$ for all $i$. This gives rise to a graded morphism $\left(\begin{smallmatrix}
-\cpx{w} &-\cpx{m}[-1]\\
0&0
\end{smallmatrix}\right): \cone(\f)\ra \cone(\g)[-1]$ such that
$$\left(\begin{matrix}
-w^{i+1}&-m^i\\
0&0
\end{matrix}\right)\left(\begin{matrix}
-d_X^{i+1}&0\\
f^{i+1}&d_Y^i
\end{matrix}\right)+\left(\begin{matrix}
-d_U^i&0\\
g^i&d_V^{i-1}
\end{matrix}\right)\left(\begin{matrix}
-w^i&-m^{i-1}\\
0&0
\end{matrix}\right)+\left(\begin{matrix}
a^i&0\\
c^i&d^i
\end{matrix}\right)=\left(\begin{matrix}
s^{i+1}&0\\
-g^iw^i+c^i&-g^im^{i-1}+d^i
\end{matrix}\right).$$
Thus $\left(\begin{smallmatrix}
\cpx{a}& 0\\
\cpx{c}& \cpx{d}
\end{smallmatrix}\right)\sim\left(\begin{smallmatrix}
\cpx{s}[1]&0\\
-\g\cpx{w}+\cpx{c}&-\g(\cpx{m}[-1])+\cpx{d}
\end{smallmatrix}\right)$. So we may assume $\cpx{a}=\cpx{s}[1]$. By (2), $\cpx{d}\f\sim \g(\cpx{a}[-1])=\g\cpx{s}$. Thanks to $\cpx{d}-\cpx{r}\sim \g\h$, we have
$$\g\h\f\sim (\cpx{d}-\cpx{r})\f=\cpx{d}\f-\cpx{r}\f\sim \g\cpx{s}-\cpx{r}\f.$$

Conversely, assume that there is a morphism $\h: \Y\ra \U$ such that $\g\h\f\sim \g\cpx{s}-\cpx{r}\f.$ Let $\cpx{d}:=\cpx{r}+\g\h: \Y\ra \V$. Then $\cpx{d}\f=(\cpx{r}+\g\h)\f=\cpx{r}\f+\g\h\f\sim\g\cpx{s}$. Let $\cpx{c}$ be a homotopy for $\cpx{d}\f\sim\g\cpx{s}$. By (2),
$$\cpx{\phi}:=\left(\begin{matrix}
\cpx{s}[1]&0 \\
\cpx{c}[1]&\cpx{d}
\end{matrix}\right): \cone(\f)\lra \cone(\g)$$
 is a morphism of complexes. Clearly, $\pi(\g)\cpx{\phi}=(\cpx{s}[1]) \pi(\f)$. Since $\cpx{\phi} \theta(\f)=\left(\begin{smallmatrix}
0\\
\cpx{d}
\end{smallmatrix}\right)=\theta(\g) \cpx{d}$ and $\theta(\g) \g\sim 0$, we get
$$\cpx{\phi}\theta(\f)-\theta(\g) \cpx{r}=\theta(\g) \cpx{d}-\theta(\g) \cpx{r}=\theta(\g)(\cpx{d}-\cpx{r})= \theta(\g) \g\h\sim 0.$$

(4) Since each of $\cpx{\phi}$ and $\cpx{\psi}: \cone(\f)\ra \cone(\g)$ makes the two squares in $(\divideontimes)$ commute in $\Kc$, we have the following commutative diagram in $\Kc$:
\[\begin{array}{c}\xymatrix@R=1.5em{
\X\ar[r]^{\f}&\Y\ar[r]^-{\theta(\f)}\ar[d]^{0}& \cone(\f)\ar[r]^{\pi(\f)}
\ar[d]^{\cpx{\phi}-\cpx{\psi}}&\X[1]\ar[d]^{0} \\
\U\ar[r]^{\g}&\V\ar[r]^-{\theta(\g)}& \cone(\g)\ar[r]^{\pi(\g)}&\U[1].
}\end{array}\]
Thus the desired result follows from (3).

(5) By (3), there exists a morphism $\h: \Y\ra \U$ of complexes such that $\g\h\f\sim \g\cpx{s}-\cpx{r}\f$ with $\cpx{c}$ a homotopy, and
$$\cpx{\phi}=\left(\begin{matrix}
\cpx{s}[1]&0 \\
\cpx{c}[1]&\cpx{r}+\g\h
\end{matrix}\right): \cone(\f)\lra \cone(\g)$$
in $\Kc$. Then the diagram
 \[\begin{array}{c}\xymatrix@R=1.5em{
\X\ar[r]^{\f}\ar[d]^{\s}&\Y\ar[r]^-{\theta(\f)}\ar[d]^{\cpx{r}+\g\h}& \cone(\f)\ar[r]^{\pi(\f)}
\ar[d]^{\cpx{\phi}}&\X[1]\ar[d]^{\cpx{s}[1]} \\
\U\ar[r]^{\g}&\V\ar[r]^-{\theta(\g)}& \cone(\g)\ar[r]^{\pi(\g)}&\U[1]
}\end{array}\]
is commutative in $\Kc$. Since $\cpx{r}$ and $\cpx{s}$ are isomorphisms and $\g$ is a radical morphism in $\Kc$, both $\s$ and $\cpx{r}+\g\h$ are isomorphisms in $\Kc$. It follows from the above diagram that $\cpx{\phi}$ is an isomorphism in $\Kc$.
\end{proof}

\subsection{Commutative Frobenius algebras}\label{CFA}

Throughout this subsection, $A$ denotes a finite-dimensional local commutative Frobenius algebra over a field $k$, and $\omega:A\ra k$ is a Frobenius form of $A$.

Let $U, V$ be finite-dimensional $k$-spaces. Assume that there is a non-degenerate $k$-bilinear form $\lan -,-\ran: U\times V\ra k$. Let $W$ be a subspace of $U$. Define
$$W^{\perp}=\{v\in V \mid \lan w, v\ran =0, \forall w\in W\}.$$
Then ${}^{\perp}(W^{\perp})=W$.
Assume that there is another non-degenerate $k$-bilinear form $\lan -,-\ran': V\times U\ra k$ such that $\lan v, u\ran' = \alpha\lan u, v\ran$ for $u\in U, v\in V$, where $\alpha\in k \setminus \{0\}$ is a fixed element. Then $W^{\perp} = {}^{\perp}W$ with respect to the two bilinear forms, respectively. In this case, we just write $W^{\perp}$.

Let $X$ be a free $A$-module of finite rank and $f\in \End_A(X)$. We define $\tr(f)$ to be the trace of the matrix of $f$ with respect to a chosen basis of $X$. If $X=0$, then we define $\tr(f)=0$. Clearly, the  function $\tr$ is $k$-linear and  independent of the choice of basis. For two  free $A$-modules $X$ and $Y$ of finite rank, one deduces from the commutativity of $A$ that $\tr(fg)=\tr(gf)$ for all $f\in\Hom_A(X,Y)$ and $g\in\Hom_A(Y,X)$. Since $A$ is a local algebra, every finitely generated projective $A$-module is isomorphic to $A^m$ for some $m\geq 0$. Thus the function $\tr$ is defined for $\End_A(X)$ for every finitely generated projective $A$-module $X$.

 \medskip
Let $\X$ and $\Y$ be complexes in $\Kb{A}$. The Hom-complex $\Hom^{\bullet}(\X, \Y)$ of $\X$ and $\Y$ is defined as follows. For $n\in \mathbb{Z}$, its $n$-th term is
$$\Hom^n(\X, \Y):=\prod_{p\in \mathbb{Z}}\Hom_A(X^p, Y^{p+n}),$$
and the differential is given by
$$d^n:  \Hom^n(\X, \Y)\lra \Hom^{n+1}(\X, \Y),$$
$$ (f^p: X^p\ra Y^{p+n})_{p\in \mathbb{Z}}\longmapsto \big(d_Y^{n+p}f^p+(-1)^{n+1}f^{p+1}d_X^p: X^p\ra Y^{p+n+1}\big)_{p\in \mathbb{Z}}.$$
As $\X$ and $\Y$ are bounded complexes, $\Hom^n(\X, \Y)$ is a finite-dimensional $k$-space for all $n\in \mathbb{Z}$. For each $n\in \mathbb{Z}$, define the $n$-\emph{cocycle} and $n$-\emph{coboundary} modules of the Hom-complex  by
$$Z^n(\Hom^{\bullet}(\X, \Y)):=\ker(d^n),\;\;\; B^n(\Hom^{\bullet}(\X, \Y)):=\im(d^{n-1}),  \mbox{ respectively. }$$

The following result may be known. For the convenience of the reader, we include here a proof.
\begin{lemma}\label{form}
Keep the notation above, and let $n\in\mathbb{Z}$. Then
 $$ \lan -,-\ran: \Hom^n(\X, \Y)\times \Hom^{-n}(\Y, \X)   \lra  k, $$
$$(\f,\g)  \longmapsto \sum_{p\in \mathbb{Z}}(-1)^p\omega(\tr(f^{p-n}g^p)) \; \mbox{ for } \f\in  \prod_{p\in \mathbb{Z}}\Hom_A(X^p, Y^{p+n}), \, \g\in \prod_{p\in \mathbb{Z}}\Hom_A(Y^p, X^{p-n})$$
is a non-degenerate $k$-bilinear form, satisfying the following properties.

$(1)$ $\lan \f, \g \ran=(-1)^n\lan \g, \f \ran$ for  $\f\in \Hom^n(\X, \Y)$ and $\g\in \Hom^{-n}(\Y, \X)$.

$(2)$ $\lan d^{n-1}\cpx{h}, \cpx{z}\ran =(-1)^{n}\lan \cpx{h}, d^{-n}\cpx{z}\ran$ for $\cpx{h}\in \Hom^{n-1}(\X, \Y)$ and $\cpx{z}\in \Hom^{-n}(\Y, \X)$.

$(3)$ Let $(-)^\perp$ be the orthogonal subspace with respect to the bilinear form $\lan -,-\ran$. Then
 \[B^n(\Hom^{\bullet}(\X, \Y))^{\perp}=Z^{-n}(\Hom^{\bullet}(\Y, \X)) \mbox{ and }
 Z^{-n}(\Hom^{\bullet}(\Y, \X))^{\perp}=B^n(\Hom^{\bullet}(\X, \Y)).\]
\end{lemma}
\begin{proof}
Since $\X$ and $\Y$ are bounded,  the summation in $\lan \f, \g \ran$ is a finite sum. It is evidently $k$-bilinear. Now if $0\ne \g\in\Hom^{-n}(\Y,\X)$, then there is some $p\in\Z$ such that $0\ne g^p: Y^p\ra X^{p-n}$. We may assume  $Y^p\simeq A^r$ and $X^{p-n}\simeq A^s$ for positive integers $r,s$. Then $g^p$ can be viewed as an $s\times r$-matrix over $A$. Suppose that the $(i,j)$-entry $a$ of $g^p$ is nonzero. Since $\omega$ is a Frobenius form of $A$, there is $b\in A$ such that $\omega(ba)\neq 0$. Let $h$ be an $r\times s$-matrix over $A$ with $(j,i)$-entry $b$ and other entries zero. Then $\omega(\tr(hg^p))=\omega(ba)\neq 0$. The matrix $h$ can be viewed as an $A$-module homomorphism from $X^{p-n}$ to $Y^p$. Let $\f\in\Hom^n(\X,\Y)$ be such that $f^j=0$ for all $j\neq p-n$ and  $f^{p-n}=h$. Then $\lan \f, \g\ran=(-1)^p\omega(\tr(hg^p))\neq 0$. Similarly, for any $0\ne \f\in\Hom^n(\X,\Y),$ we can find $\g\in \Hom^{-n}(\Y,\X)$ such that $\lan\f,\g\ran\ne 0$. Therefore, the bilinear form $\lan-,-\ran$ is non-degenerate.

To see (1), let $\f\in\Hom^n(\X,\Y)$ and $\g\in\Hom^{-n}(\Y,\X)$. Setting $q=p-n$, we obtain
\[
\begin{aligned}
\lan\f,\g\ran
&=\sum_{p\in\mathbb Z}(-1)^p\omega\bigl(\tr(f^{p-n}g^p)\bigr) =\sum_{p\in\mathbb Z}(-1)^p\omega\bigl(\tr(g^pf^{p-n})\bigr)\\
&=\sum_{q\in\mathbb Z}(-1)^{q+n}\omega\bigl(\tr(g^{q+n}f^q)\bigr)
=(-1)^n\lan\g,\f\ran.
\end{aligned}
\]

We prove (2). Pick up $\cpx{h}\in \Hom^{n-1}(\X, \Y)$ and $\cpx{z}\in \Hom^{-n}(\Y, \X)$. By definition,
\[
(d^{n-1}\cpx{h})^{p-n}=d_Y^{p-1}h^{p-n}+(-1)^nh^{p-n+1}d_X^{p-n} \text{ and }
(d^{-n}\cpx{z})^p=d_X^{p-n}z^p+(-1)^{1-n}z^{p+1}d_Y^p.
\]
Consequently,
\[
\begin{aligned}
\lan d^{n-1}\cpx{h},\cpx{z}\ran
&=\sum_{p\in\mathbb Z}(-1)^p\omega\Bigl(\tr\bigl(
  (d_Y^{p-1}h^{p-n}+(-1)^nh^{p-n+1}d_X^{p-n})z^p\bigr)\Bigr)\\
&=\sum_{p\in\mathbb Z}(-1)^{p+1}\omega\bigl(\tr(h^{p-n+1}z^{p+1}d_Y^p)\bigr)
  +\sum_{p\in\mathbb Z}(-1)^{p+n}\omega\bigl(\tr(h^{p-n+1}d_X^{p-n}z^p)\bigr)\\
&=(-1)^n\sum_{p\in\mathbb Z}(-1)^p\omega\Bigl(\tr\bigl(
  h^{p-n+1}(d_X^{p-n}z^p+(-1)^{1-n}z^{p+1}d_Y^p)\bigr)\Bigr)\\
&=(-1)^n\lan\cpx{h},d^{-n}\cpx{z}\ran.
\end{aligned}
\]
Here, we have shifted the index by one in the second equality. This proves (2).

Finally, we show (3). Let $\cpx{z}\in\Hom^{-n}(\Y,\X)$. By (2), the element $\cpx{z}$ is orthogonal to $B^n(\Hom^{\bullet}(\X,\Y))$ if and only if $\lan\cpx{h},d^{-n}\cpx{z}\ran=0$ for every $\cpx{h}\in\Hom^{n-1}(\X,\Y)$. This is equivalent to saying that $d^{-n}\cpx{z}=0$ because $\lan-,-\ran$ is non-degenerate.  Hence $B^n(\Hom^{\bullet}(\X,\Y))^{\perp}=Z^{-n}(\Hom^{\bullet}(\Y,\X))$. Further, taking orthogonal complements, we get
\[B^n(\Hom^{\bullet}(\X,\Y))=\big(B^n(\Hom^{\bullet}(\X,\Y))^{\perp}\big)^{\perp} = Z^{-n}(\Hom^{\bullet}(\Y,\X))^{\perp}.\qedhere\]
\end{proof}

\section{A characterization of standard derived equivalences}\label{aut}
In this section, we study general properties of derived equivalences between finite-dimensional algebras and prove Theorem \ref{t0}.

Let $A$ and $B$ be finite-dimensional algebras over a common field, and let $F: \Db{A}\ra \Db{B}$ be a derived equivalence. It was proved in \cite{R} that  there exists a standard derived equivalence $\w{F}: \Db{A}\ra \Db{B}$ such that $\w{F}$ and $F$ agree on the subcategory $A\proj$ and $\w{F}(\X)\simeq F(\X)$ for every object $\X$ in $\Db{A}$. Determining whether $F$ is standard is equivalent to understanding the natural-isomorphism class of the difference $\w{F}^{-1}F: \Db{A}\ra \Db{A}$. Our aim is to find a particularly convenient representative of this class.

We proceed in two steps. Precisely speaking, a \emph{derived equivalence} $F: \Db{A}\ra \Db{B}$ is a pair $(F, \xi): \Db{A}\ra \Db{B}$ consisting of an equivalence functor $F$ and a connecting isomorphism $\xi: F\circ [1]\ra [1]\circ F$ (the definition of triangle functors is recalled below). First, we treat the connecting isomorphism. We show that, up to natural isomorphism of triangle functors, the connecting isomorphism can be taken as identity (see Proposition \ref{id}). Second, we consider a special class of triangles---the truncated triangles---and show that, up to natural isomorphism of triangle functors, the difference $\w{F}^{-1}F$ fixes all objects, subcategories $A\md[i]$ for $i\in \Z$, and truncated triangles. The two steps together will yield Theorem \ref{t0}.

\medskip
Let us start by recalling the definition of triangle functors.

Let $(\C, \Sigma)$ and $(\A, \Sigma')$ be triangulated categories with suspension functors $\Sigma$ and $\Sigma'$, respectively. A \emph{triangle functor} from $\C$ to $\A$ is a pair $(F,\phi)$ consisting of an additive functor $F:\C\ra\A$ and a natural isomorphism $\phi:F\Sigma\ra\Sigma'F$, such that triangles are preserved: whenever
$X\raf{f}Y\raf{g}Z\raf{h}\Sigma(X)$
is a triangle in $\C$, the sequence
$F(X)\lraf{F(f)}F(Y)\lraf{F(g)}F(Z)\lraf{\phi_XF(h)}\Sigma' (F(X))$
is a triangle in $\A$.   The natural isomorphism $\phi$ is called a \emph{connecting isomorphism} of $F$. If $F$ is an equivalence, then $(F,\phi)$ is called a \emph{triangle equivalence}.
Given another triangle functor $(G,\psi):\C\ra\A$, a \emph{natural transformation} $\mu:(F,\phi)\ra(G,\psi)$ of triangle functors is a natural transformation $\mu:F\ra G$ compatible with the connecting isomorphisms; explicitly, the following diagram must commute:
  \[\begin{array}{c}\xymatrix@R=1.5em{
F\Sigma\ar[r]^{\phi}\ar[d]_{\mu\, \Sigma} &\Sigma' F\ar[d]^{\Sigma' \mu}\\
G\Sigma\ar[r]^{\psi}&\Sigma' G.
}\end{array}\]

We next introduce a useful modification of a triangle functor. Suppose that $(F,\phi):\C\ra\A$ is a triangle functor. Let $\B$ be a full subcategory of $\C$, and choose a family of automorphisms
$$\mu_{\B}=\{\mu_{X}: F(X)\ra F(X)\;|\; X\in \B\}$$
in $\A$. Extend this family to all objects of $\C$ by setting $\mu_Z:=\id_{F(Z)}$ whenever $Z$ is not in $\B$. Define a functor $G:\C\ra\A$ by
$$G(X):=F(X),\quad G(f):=\mu_YF(f)\mu_X^{-1}$$
for every object $X$ and morphism $f:X\ra Y$ in $\C$. A natural isomorphism $\psi:G\Sigma\ra\Sigma'G$ of functors can be defined by
$$\psi_X=\Sigma'(\mu_X)\phi_X\mu_{\Sigma(X)}^{-1}:G(\Sigma(X))\lra\Sigma'(G(X)).$$
Then $(G,\psi)$ is a triangle functor and the family $\mu$ defines a natural isomorphism $(F,\phi)\ra(G,\psi)$ of triangle functors. We call $(G,\psi)$ the {\it conjugate functor} of $(F,\phi)$ by $\mu_{\B}$.

\medskip
We say that $\C$ is {\it $\Sigma$-free} if $\Sigma^i(X)\not\simeq X$ for any nonzero object $X$ of $\C$ and any $i\ne 0$. For an object $X$ in $\C$, we denote its isomorphism class by $[X]$. For a subcategory $\B$ of $\C$, we denote by $[\B]$ the collection of all isomorphism classes of objects of $\B$. We say that $\C$ is {\it essentially small} if $[\C]$ is a set.

Assume that $\C$ is essentially small. Then the group generated by $\Sigma$ acts on $[\C]$. A subset $\mathcal{U}$ of $[\C]$ is called a {\it $\Sigma$-section} of $[\C]$ if $\mathcal{U}$ consists of representatives of all orbits of this action on $[\C]$.

Let  $F: \C\ra \C$ be a triangle functor of $\C$. An object $X$ in $\C$ (respectively, a morphism $f: Y\to Z$ in $\C$) is \emph{fixed by $F$} if $F(X)=X$ (respectively, $F(f)=f$). We say that $F$ \emph{fixes} a full subcategory $\B$ (respectively, a triangle $\Delta$ or an ideal $\mathcal{I}$) of $\C$ if $F$ fixes all objects and morphisms in $\B$ (respectively, all objects and morphisms in $\Delta$ or all morphisms in $\mathcal{I}$).

\begin{lemma}\label{connect}
Let $(\C, \Sigma)$ and $(\A, \Sigma')$ be triangulated categories. Assume that $(\C, \Sigma)$ is essentially small and $\Sigma$-free.  Let $(F, \phi): \C\ra \A$ be a triangle functor such that $F(\Sigma(X))=\Sigma' (F(X))$ for every object $X$ of $\C$. Let $\mathcal{U}$ be a $\Sigma$-section of $[\C]$. Then there is a triangle functor $(F_{\mathcal{U}}, \id): \C\ra \A$ with respect to $\mathcal{U}$ such that $(F, \phi)\simeq (F_{\mathcal{U}}, \id)$ as triangle functors.

Moreover, let $\B$ be a full subcategory of $\C$ such that $[\B]\subseteq \mathcal{U}$. Assume that $(\C, \Sigma)=(\A, \Sigma')$ and $F$ fixes $\Sigma^i(\B)$ for all $i\in \Z$. Then $F_{\mathcal{U}}$ also fixes $\Sigma^i(\B)$ for all $i\in \Z$.
\end{lemma}
\begin{proof}
Since $\C$ is $\Sigma$-free, we know that, for every object $Z\ne 0$ in $\C$, there is a unique $n\in \mathbb{Z}$ such that $[Z]\in \Sigma^n(\mathcal{U}):=\{[\Sigma^{n}(Y)]\,|\,[Y]\in \mathcal{U}\}$.

For $n\le 0$ and any object $X$ in $\C$ with $[X]\in \Sigma^n(\mathcal{U})$, we construct an automorphism $\mu_X: F(X)\ra F(X)$.
This is done by induction on $n$.

For $n=0$, let $\mu_X=\id_{F(X)}$ for object $X$ with $[X]\in \mathcal{U}$. Assume that we have constructed $\mu_X$ for all objects $X$ with $[X]\in \Sigma^n(\mathcal{U})$. Now, let $X$ be an object of $\C$ such that $[X]\in \Sigma^{n-1}(\mathcal{U})$. Then $[\Sigma(X)]\in \Sigma^n(\mathcal{U})$. Let $\mu_X:=(\Sigma')^{-1}(\mu_{\Sigma(X)}\circ\phi_{X}^{-1})$. Then we have constructed $\mu_X$ for all objects $X$ with $[X]\in \bigcup_{-\infty<i\leq 0}\Sigma^i(\mathcal{U})$. Furthermore, we have the commutative diagram
  \[\begin{array}{c}\xymatrix{
F(\Sigma(X))\ar[r]^{\phi_{X}}\ar[d]_{\mu_{\Sigma(X)}} &\Sigma' (F(X))\ar[d]^{\Sigma'(\mu_X)}\\
F(\Sigma(X))\ar[r]^{\id}&\Sigma' (F(X)).
}\end{array}\]
 Similarly, we can construct $\mu_X$ for all $X\in \C$ with $[X]\in \bigcup_{0<i<+\infty}\Sigma^i(\mathcal{U})$. Denote the resulting family of automorphisms by $\mu_{\C}$. Let $(F_{\mathcal{U}}, \psi): \C\ra \A$ be the conjugate functor of $(F, \phi)$ by $\mu_{\C}$. By the construction of $\mu_{\C}$, we have $\psi=\id$.

Now we prove the second statement. We prove that $F_{\mathcal{U}}$ fixes $\Sigma^n(\B)$ by induction on $n\leq 0$. Since $[\B]\subseteq \mathcal{U}$ and we have chosen $\mu_X=\id_X$ for $X$ with $[X]\in \mathcal{U}$, we know clearly that $F_{\mathcal{U}}$ fixes $\Sigma^0(\B)$. Assume that $F_{\mathcal{U}}$ fixes $\Sigma^n(\B)$. Let $f: X\ra Y$ be a morphism in $\Sigma^{n-1}(\B)$. Then $ \Sigma (f)$ is a morphism in $\Sigma^n(\B)$, hence $F_{\mathcal{U}}(\Sigma (f))=\Sigma (f)$. Since $\id: F_{\mathcal{U}}\Sigma\ra \Sigma F_{\mathcal{U}}$ is a natural isomorphism, $\Sigma (F_{\mathcal{U}}(f))=F_{\mathcal{U}}( \Sigma (f))= \Sigma (f)$. So $F_{\mathcal{U}}$ fixes $f$.

For the case $\Sigma^n(\B)$ with $n>0$, the proof can be done similarly.
\end{proof}

\medskip
Let $A$ be an Artin algebra. Following \cite{CY}, a triangle functor $(F, \xi): \Db{A}\ra \Db{A}$ (respectively, $\Kb{A}\ra \Kb{A}$) is called a {\it pseudo-identity} if the following two conditions hold:
\begin{enumerate}
\item $F$ fixes the full subcategory $A\md[i]$ (respectively, $A\proj[i]$) for each $i\in \Z$;
\item $F$ fixes all objects of $\Db{A}$ (respectively, $\Kb{A}$).
\end{enumerate}

\medskip
The following result is due to \cite[Proposition 5.8]{CY}, motivated by \cite[Corollary 3.5]{R}. Note that the proof of (2) is similar to that of (1).

\begin{lemma}{\rm\cite{CY}}\label{group}
Suppose that $A$ and $B$ are finite-dimensional algebras over a common field.

$(1)$ Let $(F, \xi): \Db{A}\ra \Db{B}$ be a derived equivalence. Then there is a standard derived equivalence $(G, \psi): \Db{A}\ra \Db{B}$ and a pseudo-identity $(I, \phi): \Db{A}\ra \Db{A}$ such that $(F, \xi)\simeq(G, \psi)\circ(I, \phi)$ as triangle functors.

$(2)$ Let $(F, \xi): \Kb{A}\ra \Kb{B}$ be a derived equivalence. Then there is a standard derived equivalence $(G, \psi): \Kb{A}\ra \Kb{B}$ and a pseudo-identity $(I, \phi): \Kb{A}\ra \Kb{A}$ such that $(F, \xi)\simeq (G, \psi)\circ (I, \phi)$ as triangle functors.
\end{lemma}

Derived equivalences have the following reduction.

\begin{prop}\label{id}
Suppose that $A$ and $B$ are finite-dimensional algebras over a common field. Then
every derived equivalence $(F, \xi): \Db{A}\ra \Db{B}$ (respectively, $\Kb{A}\ra \Kb{B}$) is naturally isomorphic to a derived equivalence
$(F', \id): \Db{A}\ra \Db{B}$ (respectively, $\Kb{A}$ $\ra \Kb{B}$) as triangle functors. Moreover, if $(F, \xi): \Db{A}\ra \Db{B}$
(respectively, $\Kb{A}\ra \Kb{B}$) is a pseudo-identity, then so is $(F', \id)$.
\end{prop}

\begin{proof}
Assume that $(F, \xi): \Db{A}\ra \Db{A}$ is a pseudo-identity. Clearly, $F(\X[1])=F(\X)[1]$ for every object $\X$ of $\Db{A}$. Note that $\Db{A}$ is an essentially small category and $[1]$-free. Consider $\mathcal{U}:=\{[0]\}\cup \{[\X]\in [\Db{A}]\,|\, H^0(\X)\neq 0 \text{ and } H^i(\X)=0,\, \forall i<0\}$, where $H^i$ is the $i$-th cohomology functor. Clearly, $\mathcal{U}$ is a $[1]$-section of $[\Db{A}]$. Let $\B:=A\md$. Then $[\B]\subseteq \mathcal{U}$. It follows from Lemma \ref{connect} that $(F, \xi)$ is naturally isomorphic to $(F_{\mathcal{U}}, \id)$ and $F_{\mathcal{U}}$ fixes $A\md[i]$ for all $i\in\mathbb{Z}$. By definition, $F_{\mathcal{U}}(\X)=F(\X)=\X$ for $\X\in \Db{A}$. Thus $(F_{\mathcal{U}}, \id)$ is a pseudo-identity. Let $F':=F_{\mathcal{U}}$. Then $(F',\id)$ is a desired derived equivalence.

Every standard equivalence $(G, \psi):\Db{A}\ra\Db{B}$ is naturally isomorphic to a derived tensor functor $G':=\X\OT_A-$ for some complex $\X$ of $B$-$A$-bimodules. Since every derived tensor functor commutes with the shift functor, we have $(G, \psi)\simeq (G', \id)$ as triangle functors.

It follows from Lemma \ref{group}(1) that, for every derived equivalence $(F, \xi): \Db{A}\ra \Db{B}$, there is a triangle functor $(F',\id)$ such that $(F,\xi)\simeq (F', \id): \Db{A}\ra \Db{B}$ as triangle functors.

The proof for the derived equivalence $(F,\xi): \Kb{A}\ra \Kb{B}$ can be done similarly.
\end{proof}

Let $\X$ be a complex. For each $i\in \Z$, we define a truncated complex ${\X}^{\geq i}$ of $\X$ by
$${\X}^{\geq i}:\;\; \cdots \lra 0\lra X^i\lraf{d_X^i} X^{i+1}\lraf{d_X^{i+1}} X^{i+2}\lra \cdots $$
with $X^i$ in degree $i$. Similarly, we define ${\X}^{\leq i}$, and ${\X}^{[u, v]}=({\X}^{\geq u})^{\bullet\leq v}$ for integers $u$ and $v$ with $u\leq v$.

Let $\f: \X\ra\Y$ be a morphism of complexes. We define a morphism ${\cpx{f}}^{\geq i}: {\X}^{\geq i}\ra {\Y}^{\geq i}$ of complexes by $({\cpx{f}}^{\geq i})^j:=f^j$ for all $j\geq i$. Similarly, we define ${\cpx{f}}^{\leq i}: {\X}^{\leq i}\ra {\Y}^{\leq i}$ and ${\cpx{f}}^{[u, v]}: {\X}^{[u, v]}\ra {\Y}^{[u, v]}$ for $u\leq v$.

\begin{defn}\label{def-truncated}
Let $A$ be an Artin algebra, and let $\X$ be a nonzero complex in $\Kb{A}$. The {\it truncated triangle of $\X$ in degree $i$} is defined to be the standard triangle
$$\Delta_i(\X):\;\; {\X}^{\leq i}[-1]\lraf{\uu{d_X^i}} {\X}^{\geq i+1}\lraf{\theta(\uu{d_X^i})} \X\lraf{\pi(\uu{d_X^i})} {\X}^{\leq i}$$
in $\Kb{A}$, where $\uu{d_X^i}$ is the morphism with the $(i+1)$-component being $d_X^i$.

We write $\Delta_L(\X)$ for $\Delta_m(\X)$, and $\Delta_R(\X)$ for $\Delta_{n-1}(\X)$, where $m$ is the smallest integer such that $X^m\neq 0$ and $n$ is the largest integer such that $X^n\neq 0$. In this case, $n-m$ is called the {\em length} $\ell(\X)$ of $\X$, while $\Delta_L(\X)$ and $\Delta_R(\X)$ are called the \emph{left and right} truncated triangles of $\X$, respectively.
For convenience, we write $\theta_L(\X):=\theta(\uu{d_X^m}), \pi_L(\X):=\pi(\uu{d_X^m}), \theta_R(\X):=\theta(\uu{d_X^{n-1}})$ and $\pi_R(\X):=\pi(\uu{d_X^{n-1}})$.
\end{defn}

For our later considerations, we introduce the following ideals and subcategories of $\Kb{A}$ for an Artin algebra $A$. Let $i$ be an integer and $n$ a non-negative integer.
\begin{enumerate}
  \item $\Zmap_i(A)$ is the ideal of $\Kb{A}$ consisting of homotopy classes of all morphisms $\f:\X\ra\Y$ of complexes such that $f^t=0$ for all $t\neq i$ and $f^i$ factorizes through $d_Y^{i-1}$;
  \item $\K^{\ell\leq n}(A\proj)$ is the full subcategory of $\Kb{A}$ consisting of all complexes of length at most $n$; and
  \item $\K^{0,n}(A\proj)$ is the full subcategory of $\Kb{A}$ consisting of all complexes of length $n$ whose nonzero terms are concentrated in degrees $0,1,\ldots,n$.
\end{enumerate}

Let $\B$ be a full subcategory of $\Kb{A}$ closed under the shift functor $[1]$. A family of automorphisms $\lambda_{\B}=\{\lambda_{\X}:\X\ra \X\;|\; \X\in \B\}$ in $\Kb{A}$ is {\it $[1]$-compatible} if $\lambda_{\X[1]}=\lambda_{\X}[1]$ for all objects $\X$ in $\B$. A crucial property we shall use frequently is that the conjugate functor $(F', \phi)$ of a pseudo-identity
$(F, \id): \Kb{A}\ra \Kb{A}$ by a $[1]$-compatible family $\lambda_{\B}$ also satisfies $\phi=\id: F'[1]\ra [1]F'$.

\begin{prop}\label{truncated}
Let $A$ be a finite-dimensional algebra, and let $(F, \phi): \Kb{A}\ra \Kb{A}$ be a pseudo-identity. Then there is a pseudo-identity $(G, \id): \Kb{A}\ra \Kb{A}$ such that $(F, \phi)\simeq (G, \id)$ as triangle functors and $G$ fixes all truncated triangles in $\Kb{A}$.
\end{prop}

\begin{proof}
By Proposition \ref{id}, we may assume $\phi=\id$.
We will construct, by induction on the length $n$ of complexes, a $[1]$-compatible family of automorphisms  by which the conjugate functor $(G,\id)$ of $(F,\phi)$ is a pseudo-identity and fixes both left and right truncated triangles.

Let $n=0$. For any complex $\X$ in $\K^{\ell\leq 0}(A\proj)$, we define $\lambda_{\X}=\id_{\X}$ and denote by $\lambda_{\K^{\ell\leq 0}(A\proj)}$ the family of these automorphisms. By the construction, the conjugate functor $(F_0,\id)$ of $(F, \id)$ by $\lambda_{\K^{\ell\leq 0}(A\proj)}$ is just $(F, \id)$, hence it is a pseudo-identity. Since complexes in $\K^{\ell\leq 0}(A\proj)$ are direct sums of shifts of stalk complexes, it is clear that $(F_0,\id)$ fixes the left and right truncated triangles of complexes in $\K^{\ell\leq 0}(A\proj)$.

Suppose that $n > 0$ and a $[1]$-compatible family $\lambda_{\K^{\ell\leq n-1}(A\proj)}$ has been constructed on objects of length at most $n-1$ such that the conjugate functor $(F_{n-1},\id)$ of $(F_0,\id)$ by this family is a pseudo-identity and fixes both left and right truncated triangles of complexes with length at most $n-1$. Now, we extend this family to a family which contains automorphisms of   complexes of length $n$ in $\Kb{A}$ as follows.

Let $\K^{\ell=n}(A\proj)$ be the full subcategory of $\Kb{A}$ consisting of all complexes of length $n$. We first start with a complex $\cpx{X}$ in $\K^{0,n}(A\proj)$ and consider the left truncated triangle of $\cpx{X}$. By induction hypothesis, both $\pi_L({\X}^{\geq 1})\circ \uu{d_X^0}:$
\[X^0[-1]\lraf{\uu{d_X^0}}{\X}^{\geq 1}\lraf{\pi_L({\X}^{\geq 1})}X^1[-1]\]
and $\pi_L({\X}^{\geq 1})$ are fixed by $F_{n-1}$ since $F_{n-1}$ is a pseudo-identity and ${\X}^{\geq 1}$ has length less than $n$. This implies that $F_{n-1}(\uu{d_X^0})-\uu{d_X^0}: X^0[-1]\ra {\X}^{\geq 1}$ factorizes through $\theta_L({\X}^{\geq 1}): {\X}^{\geq 2}\ra {\X}^{\geq 1}$. It follows from $\Hom_{\Kb{A}}(X^0[-1],{\X}^{\geq 2})=0$  that $F_{n-1}(\uu{d_X^0})-\uu{d_X^0}=0$, that is, $F_{n-1}$ also fixes $\uu{d_X^0}$.
Thus there is an automorphism $\lambda_{\X}^L: \X\ra \X$ in $\Kb{A}$ such that the following diagram commutes:
  \[\xymatrix{
F_{n-1}(\Delta_L(\X)): &X^0[-1]\ar[r]^{\uu{d_X^0}}\ar[d]^{1}&{\X}^{\geq 1}\ar[d]^{1}\ar[rr]^-{F_{n-1}(\theta_L(\X))} &&\X\ar@{..>}[d]^{\lambda_{\X}^L}\ar[rr]^{F_{n-1}(\pi_L(\X))}&&X^0\ar[d]^{1} \\
\Delta_L(\X): &X^0[-1]\ar[r]^{\uu{d_X^0}}&{\X}^{\geq 1}\ar[rr]^-{\theta_L(\X)} &&\X\ar[rr]^{\pi_L(\X)}&&X^0.}\]
in $\Kb{A}$.

For an arbitrary complex $\Y$ in $\Kb{A}$ of length $n$, there is a unique integer $i$ such that $\Y[i]$ lies in $\K^{0,n}(A\proj)$. Then we define $\lambda_{\Y}^L=\lambda_{\Y[i]}^L[-i]$ and get a $[1]$-compatible family $\lambda^L_{\K^{\ell=n}(A\proj)}$ of automorphisms in $\K^{\ell=n}(A\proj)$.  The conjugate functor $(F^L_n, \id):\Kb{A}\to \Kb{A}$ of $(F_{n-1}, \id)$ by $\lambda^L_{\K^{\ell=n}(A\proj)}$ is a pseudo-identity, fixing both left and right truncated triangles of complexes with length at most $n-1$. Moreover, the above diagram implies that $F^L_n$ fixes $\Delta_L(\X)$ for all complexes $\X$ in $\K^{0,n}(A\proj)$. Thanks to $F_n^L[1]=[1]F_n^L$, we know that $F^L_n$ fixes $\Delta_L(\X)$ for all complexes $\X$ in $\K^{\ell=n}(A\proj)$.

\smallskip
We now further conjugate $F_n^L$ by another $[1]$-compatible family of automorphisms of complexes of length $n$.  Let $\X$ be an object in $\K^{0,n}(A\proj)$. First, we claim that
$$(\star)\;\;\;F^L_n(\uu{d_X^{n-1}})=\uu{d_X^{n-1}},\;\;F^L_n(\theta_R(\X))=\theta_R(\X)\;\;\text{and}\;\;  F^L_n(\pi_R(\X))=\pi_R(\X)+\cpx{\phi}$$
for some $\cpx{\phi}\in \Zmap_1(A)(\X, {\X}^{\leq n-1})$, where $\uu{d_X^{n-1}}: {\X}^{\leq n-1}[-1]\ra X^n[-n]$ is the defined morphism. Observe that
$$\uu{d_X^{n-1}}\theta_R({\X}^{\leq n-1}[-1])=d_X^{n-1}[-n]: X^{n-1}[-n]\lra X^n[-n].$$
Since $F^L_n$ fixes both $\uu{d_X^{n-1}}\theta_R({\X}^{\leq n-1}[-1])$ and $\theta_R({\X}^{\leq n-1}[-1])$, it follows that
\[(F_n^L(\uu{d_X^{n-1}})-\uu{d_X^{n-1}})\circ \theta_R({\X}^{\leq n-1}[-1])=0.\]
Hence $F_n^L(\uu{d_X^{n-1}})-\uu{d_X^{n-1}}$ factorizes through $\pi_R({\X}^{\leq n-1}[-1]): {\X}^{\leq n-1}[-1]\ra {\X}^{\leq n-2}[-1]$. However $\Hom_{\Kb{A}}({\X}^{\leq n-2}[-1], X^n[-n])=0$. Therefore $F_n^L(\uu{d_X^{n-1}})-\uu{d_X^{n-1}}=0$, and $F_n^L$ fixes $\uu{d_X^{n-1}}$. This shows the first equality of $(\star)$. Since
$$\theta_R(\X)=\theta_L({\X}^{\geq 0})\theta_L({\X}^{\geq 1})\cdots \theta_L({\X}^{\geq n-2})\theta_L({\X}^{\geq n-1})$$
and $F^L_n$ fixes $\theta_L({\X}^{\geq i})$ for all $0\leq i\leq n-1$, the second equality of $(\star)$ follows.

To prove the last equality in $(\star)$, we consider the morphism of left truncated triangles
  \[\xymatrix{
\Delta_L(\X): &X^0[-1]\ar[rr]^-{\uu{d_X^0}}\ar[d]^1&&{\X}^{\geq 1}\ar[d]^{\pi_R({\X}^{\geq 1})}\ar[rr]^-{\theta_L(\X)} &&\X\ar[d]^{\pi_R(\X)}\ar[rr]^{\pi_L(\X)}&&X^0\ar[d]^1 \\
\Delta_L({\X}^{\leq n-1}): &X^0[-1]\ar[rr]^-{\uu{d_{{\X}^{\leq n-1}}^0}}&&{\X}^{[1, n-1]}\ar[rr]^-{\theta_L({\X}^{\leq n-1})} &&{\X}^{\leq n-1}\ar[rr]^{\pi_L({\X}^{\leq n-1})}&&X^0.}\]
Note that $F^L_n$ fixes $\Delta_L(\X)$, $\Delta_L({\X}^{\leq n-1})$ and $\pi_R({\X}^{\geq 1})$. Applying $F^L_n$ to the above diagram and using Lemma \ref{st}(4), we obtain the last equality of $(\star)$ for some $\cpx{\phi}\in \Zmap_1(A)(\X, {\X}^{\leq n-1})$. This finishes the proof of $(\star)$.

By the definition of $\Zmap_1(A)$, we can represent the morphism $\cpx{\phi}$ so that $\phi^i=0$ unless $i=1$ and that there is a module homomorphism $h: X^1\ra X^0$ such that $\phi^1=d_{X^{\leq n-1}}^0h$ and  $d_{X^{\leq n-1}}^0hd_X^0=0$. Note that  $d_{X^{\leq n-1}}^0=0$ if $n=1$.

We define $\cpx{\psi}:\cpx{X}\ra \cpx{X}$ by $\psi^1=d_{X^{\leq n-1}}^0h$ and $\psi^i=0$ for all $i\neq 1$. Then $(\cpx{\psi})^2=0$, and $\mu_{\cpx{X}}=\id_{\X}+\cpx{\psi}$ is an automorphism with $\mu_{\cpx{X}}^{-1}=\id_{\X}-\cpx{\psi}$.
It is straightforward to check that $\cpx{\phi}=\pi_R(\cpx{X})\cpx{\psi}$ and $\cpx{\psi}\theta_R(\cpx{X})=0$.

Let $\Y$ be a complex in $\K^b(A\proj)$ of length $n$, and let $i$ be the unique integer such that $\Y[i]$ is in $\K^{0,n}(A\proj)$. Then we define $\mu_{\Y}=\mu_{\Y[i]}[-i]$, and get a $[1]$-compatible family of automorphisms in $\K^{\ell=n}(A\proj)$.
Let $(F_n,\id)$ be the conjugate functor of $(F_n^L,\id)$ by this $[1]$-compatible family. Clearly, $(F_n, \id)$ is a pseudo-identity, fixing both left and right truncated triangles of complexes with length at most $n-1$.

By the definition of conjugate functors, together with  the equalities in $(\star)$, we get
\[\begin{aligned}
  F_n(\pi_R(\X)) & =\id_{X^{\leq n-1}}F_{n}^L(\pi_R(\X))\mu_{\X}^{-1}
     =(\pi_R(\X)+\cpx{\phi})(\id_{\X}-\cpx{\psi})\\
    & =(\pi_R(\X)+\pi_R(\X)\cpx{\psi})(\id_{\X}-\cpx{\psi})=\pi_R(\X).
\end{aligned}\]
Similarly, we have
\[F_n(\theta_R(\X))  =\mu_{\X}\, F_{n}^L(\theta_R(\X))\,\id_{X^{\geq n}}^{-1}
   =\mu_{\X}\,\theta_R(\X)=(\id_{\X}+\cpx{\psi})\theta_R(\X)=\theta_R(\X),\]
where the second equality follows from the second equality of $(\star)$. Since both ${\X}^{\leq n-1}[-1]$ and ${\X}^{\geq n}$ have length less than $n$, we have $F_n(\uu{d_X^{n-1}})=F_{n}^L(\uu{d_X^{n-1}})=\uu{d_X^{n-1}}$. Hence  $F_n$ fixes $\Delta_R(\X)$ for all $\X\in\K^{0,n}(A\proj)$.

It remains to check that $F_n$ fixes the left truncated triangles of complexes $\X$ in $\K^{0,n}(A\proj)$. It is easy to check that
$\pi_L(\X)\cpx{\psi}=0$ and $\cpx{\psi}\theta_L(\X)=0$ in $\Kb{A}$. Therefore
\[\mu_{\X}\theta_L(\X)=(\id_{\X}+\cpx{\psi})\theta_L(\X)=\theta_L(\X)\text{ and }\pi_L(\X)\mu_{\X}=\pi_L(\X)(\id_{\X}+\cpx{\psi})=\pi_L(\X) \mbox{ in } \Kb{A}.\]
This shows that $F_n$ fixes $\Delta_L(\X)$ for $\X\in\K^{0,n}(A\proj)$.

Since the connecting isomorphism $F_n[1]\ra [1]F_n$ is $\id$, we conclude that $F_n$ fixes $\Delta_L(\X)$ and $\Delta_R(\X)$ for all complexes $\X$ of length $n$.

Denote by $\lambda_{\K^{\ell=n}(A\proj)}$ the $[1]$-compatible family of automorphisms $\mu_{\X}\lambda_{\X}^L$ for complexes $\X$ in $\K^{\ell=n}(A\proj)$,  and  by $\lambda_{\K^{\ell\leq n}(A\proj)}$ the union of $\lambda_{\K^{\ell\leq n-1}(A\proj)}$ and $\lambda_{\K^{\ell=n}(A\proj)}$. Then $(F_n, \id)$ is a conjugate functor of $(F_{n-1},\id)$ by $\lambda_{\K^{\ell=n}(A\proj)}$, and also the conjugate functor of $(F_0,\id)$ by $\lambda_{\K^{\ell\leq n}(A\proj)}$.

By passing to the union of the $[1]$-compatible families $\lambda_{\K^{\ell=n}(A\proj)}$  running over all lengths $n$ and  by forming conjugate functor, we get a pseudo-identity $(G,\id)$, naturally isomorphic to $(F, \id)$, which fixes  $\Delta_L(\X)$ and $\Delta_R(\X)$ for all complexes $\X$ in $\Kb{A}$.

We now show that $(G, \id)$ fixes the truncated triangles $\Delta_t(\X)$ for all $\X$ in $\Kb{A}$ and all $t\in \mathbb{Z}$.  Assume that there are integers $m\le n$ such that $X^m\neq 0\neq X^n$ and $X^i=0$ for $i<m$ or $i>n$. If $m=n$, then the truncated triangles are trivial and nothing needs to prove. So we may assume $m\leq t<n$.  Consider $\uu{d_X^t}: {\X}^{\leq t}[-1]\ra {\X}^{\geq t+1}$. One can check directly that
$$\pi_L({\X}^{\geq t+1})\circ \uu{d_X^t}\circ \theta_R({\X}^{\leq t}[-1])=d_X^t[-t-1]: X^t[-t-1]\lra X^{t+1}[-t-1].$$
Note that $G$ fixes $d_X^t[-t-1]$, $\theta_R({\X}^{\leq t}[-1])$ and $\pi_L({\X}^{\geq t+1})$. It follows that
\[\pi_L({\X}^{\geq t+1})\circ G(\uu{d_X^t})-\pi_L({\X}^{\geq t+1})\circ \uu{d_X^t}\]
factorizes through $\pi_R({\X}^{\leq t}[-1]):{\X}^{\leq t}[-1]\ra {\X}^{\leq t-1}[-1]$. However
\[\Hom_{\Kb{A}}({\X}^{\leq t-1}[-1],X^{t+1}[-t-1])=0.\]
Hence $\pi_L({\X}^{\geq t+1})\circ G(\uu{d_X^t})=\pi_L({\X}^{\geq t+1})\circ \uu{d_X^t}$. Consequently, $G(\uu{d_X^t})-\uu{d_X^t}$ factorizes through $\theta_L({\X}^{\geq t+1}):{\X}^{\geq t+2}\ra {\X}^{\geq t+1}$. The vanishing of $\Hom_{\Kb{A}}({\X}^{\leq t}[-1],{\X}^{\geq t+2})$ implies $G(\uu{d_X^t})=\uu{d_X^t}$, that is,  $G$ fixes $\uu{d_X^t}$. Moreover,
$$\theta(\uu{d_X^t})=\theta_L({\X}^{\geq m})\cdots \theta_L({\X}^{\geq t-1})\theta_L({\X}^{\geq t})\;\mbox{ and } \; \pi(\uu{d_X^t})=\pi_R({\X}^{\leq t+1})\cdots \pi_R({\X}^{\leq n-1})\pi_R({\X}^{\leq n}).$$
Since $G$ fixes both $\theta_L(\Y)$ and $\pi_R(\Y)$ for all objects $\Y$ in $\Kb{A}$, $G$ fixes both $\theta(\uu{d_X^t})$ and $\pi(\uu{d_X^t})$. So $G$ fixes $\Delta_t(\X)$ for $t\in\mathbb{Z}$ and $\X\in \Kb{A}$.
\end{proof}

\begin{lemma}\label{idd}
 Let $A$ be a finite-dimensional algebra over a field $k$, and let $(F, \xi): \Db{A}\ra \Db{A}$ (respectively, $\Ka\ra \Ka$, or $\Kb{A}\ra \Kb{A}$) be a triangle equivalence, such that $F$ fixes $A\proj$. If $(F, \xi)$ is standard, then $(F, \xi)$ is naturally isomorphic to $(\id, \id)$ as triangle functors.
\end{lemma}
\begin{proof}
 Since $(F, \xi)$ is standard, there is a bounded complex $\X$ of $A$-$A$-bimodules such that $(F, \xi)\simeq (\X\OT_A-, \id)$ as triangle functors. As $F$ fixes $A\proj$, the derived tensor functor $\X\OT_A-$ restricted to $A\proj$ is naturally isomorphic to the identity. This implies $\X\simeq A$ in $\Db{A^{e}}$, where $A^{e}$ is the enveloping algebra $A\otimes_kA^{op}$ of $A$. Here, $A^{op}$ stands for the opposite algebra of $A$. So $(\X\OT_A-, \id)\simeq (A\OT_A-, \id)$ as triangle functors, the latter is nothing but the identity triangle functor $(\id, \id)$. Therefore, $(F, \xi)\simeq (\id, \id)$ as triangle functors.
\end{proof}

Now we are ready to prove Theorem \ref{t0}.

\begin{proof} [{\bf Proof of Theorem \ref{t0}}]
It is known that every derived auto-equivalence of $\Db{A}$ restricts to a derived auto-equivalence of $\Kb{A}$. Let $\Out(\Db{A})$ (respectively, $\Out(\Kb{A})$) be the group of natural-isomorphism classes of derived auto-equivalences of $\Db{A}$ (respectively, $\Kb{A}$). Let $\res: \Out(\Db{A})$ $\ra \Out(\Kb{A})$ be the restriction map. Then $\res$ is a group isomorphism \cite[Proposition, p. 658]{C1}. Hence a derived auto-equivalence $F: \Db{A}\ra \Db{A}$ is standard if and only if $\res(F): \Kb{A}\ra \Kb{A}$ is standard. Therefore, Rickard's question holds for $A$ if and only if every derived auto-equivalence of $\Kb{A}$ is standard. By Lemma \ref{group}, this is equivalent to saying that every pseudo-identity of $\Kb{A}$ is standard. By Proposition \ref{truncated} and Lemma \ref{idd}, the latter is equivalent to saying that every pseudo-identity $(F, \id): \Kb{A}\ra \Kb{A}$ that fixes all truncated triangles is naturally isomorphic to the identity triangle functor $(\id,\id)$.
\end{proof}

To make Theorem \ref{t0} useful in applications, we need to understand pseudo-identities that fix all truncated triangles. First, we establish some properties of such pseudo-identities.

\begin{lemma}\label{prop-1}Let $A$ be a finite-dimensional algebra, and let $(F, \id): \Kb{A}$ $\ra \Kb{A}$ be a pseudo-identity that fixes all truncated triangles. Let
$\f=(f^i): \X\ra \Y$ be a morphism in $\Kb{A}$. If $f^i\neq 0$ for at most one integer $i$, then $F(\f)=\f$. In particular, $F$ fixes $\Zmap_j(A)$ for all $j\in \Z$.
\end{lemma}

\begin{proof} Assume that there is an integer $i$ such that $f^i\ne 0$ and $f^j=0$ for all $j\neq i$. Consider  the complex
$$\U:\;\;\cdots\lra X^{i-1}\lraf{d_X^{i-1}}X^i\lraf{f^i} Y^i\lraf{d_Y^i}Y^{i+1}\lra \cdots$$
with $U^i=X^i$.
We have the morphism $\uu{d_U^i}: {\U}^{\leq i}[-1]={\X}^{\leq i}[-1]\ra {\U}^{\geq i+1}={\Y}^{\geq i}[-1]$. Then $\f=\theta(\uu{d_Y^{i-1}})\,(\uu{d_U^i}[1])\,\pi(\uu{d_X^i})$. Since $F$ fixes all truncated triangles and $F[1]=[1]F$, we see that $F$ fixes $\theta(\uu{d_Y^{i-1}})$, $\uu{d_U^i}[1]$ and $\pi(\uu{d_X^i})$. So $F$ fixes $\f$.
\end{proof}

Let $(F, \id): \Kb{A}\ra \Kb{A}$ be a pseudo-identity that fixes all truncated triangles, and let $i$ be an integer. Suppose that  $\f: \X\ra \Y$ is a morphism in $\Kb{A}$.
We regard $\f$ as a representative of the homotopy class of $\f$ and define a graded morphism
$$\chi_{i,F}(\f): \X\lra \Y$$
by $\chi_{i,F}(\f)^{\leq i}:=\g_1$ and $\chi_{i,F}(\f)^{\geq i+1}:=\g_2$, where $\g_1$ and $\g_2$ are representatives of the homotopy classes of $F({\cpx{f}}^{\leq i})$ and $F({\cpx{f}}^{\geq i+1})$, respectively.  In the following, we show that $\chi_{i,F}(\f)$ is a well-defined morphism in $\Kb{A}$.

First, applying $F$ to the left square of the following diagram of triangles in $\Kb{A}$:
  \[(*)\;\;\;\xymatrix@C=15mm{
\Delta_i(\X): &{\X}^{\leq i}[-1]\ar[r]^-{\uu{d_X^i}}\ar[d]^{{\cpx{f}}^{\leq i}[-1]}&{\X}^{\geq i+1}\ar[d]^{{\cpx{f}}^{\geq i+1}}\ar[r]^-{\theta(\uu{d_X^i})} &\X\ar[d]^{\f}\ar[r]^{\pi(\uu{d_X^i})}&{\X}^{\leq i}\ar[d]^{{\cpx{f}}^{\leq i}} \\
\Delta_i(\Y): &{\Y}^{\leq i}[-1]\ar[r]^-{\uu{d_Y^i}}&{\Y}^{\geq i+1}\ar[r]^-{\theta(\uu{d_Y^i})} &\Y\ar[r]^{\pi(\uu{d_Y^i})}&{\Y}^{\leq i},}\] we see that $\chi_{i,F}(\f)$ is a morphism of complexes. Note that the zero-homotopy class of morphisms of complexes from ${\X}^{\leq i}[-1]$ to ${\Y}^{\geq i+1}$ has only one element $0$.  Next, it is easy to know that $\chi_{i,F}(\f)$ in $\Kb{A}$ is independent of choices of representatives of $F({\cpx{f}}^{\leq i})$ and $F({\cpx{f}}^{\geq i+1})$. Finally, we show that $\chi_{i,F}(\f)$ in $\Kb{A}$ is independent of the choices of representatives of the homotopy class $\f$. Indeed, since $\chi_{i,F}(\f)$ is additive in $\f$, it is enough to prove that $\chi_{i,F}(\f)=0$ in $\Kb{A}$ if $\f\sim 0$ as morphisms of complexes.  Assume $\f\sim 0$. Then there is a homomorphism $h: X^{i+1}\ra Y^i$ such that ${\f}^{\leq i}\sim (hd_X^i)$ and $f^{\geq i+1}\sim (d_Y^ih)$, where the morphism $(hd_X^i): {\X}^{\leq i}\ra {\Y}^{\leq i}$ is defined by the $i$-component equal to $hd_X^i$ and other components equal to $0$, and where $(d_Y^ih): {\X}^{\geq i+1}\ra {\Y}^{\geq i+1}$ is defined similarly. By Lemma \ref{prop-1}, $F({\cpx{f}}^{\leq i})=F(hd_X^i)=(hd_X^i)$ and $F({\cpx{f}}^{\geq i+1})=F(d_Y^ih)=(d_Y^ih)$. This implies that $\chi_{i,F}(\f)\sim 0$. So $\chi_{i,F}(\f): \X\ra \Y$ is a well-defined morphism in $\Kb{A}$.

Moreover, $\chi_{i,F}(\f)$ makes the following diagram in $\Kb{A}$ commute:

   \[(**)\;\;\xymatrix@C=15mm{
\Delta_i(\X): &{\X}^{\leq i}[-1]\ar[r]^-{\uu{d_X^i}}\ar[d]^{F({\cpx{f}}^{\leq i})[-1]}&{\X}^{\geq i+1}\ar[d]^{F({\cpx{f}}^{\geq i+1})}\ar[r]^-{\theta(\uu{d_X^i})} &\X\ar[d]^{\chi_{i,F}(\f)}\ar[r]^{\pi(\uu{d_X^i})}&{\X}^{\leq i}\ar[d]^{F({\cpx{f}}^{\leq i})} \\
\Delta_i(\Y): &{\Y}^{\leq i}[-1]\ar[r]^-{\uu{d_Y^i}}&{\Y}^{\geq i+1}\ar[r]^-{\theta(\uu{d_Y^i})} &\Y\ar[r]^{\pi(\uu{d_Y^i})}&{\Y}^{\leq i}.}\]

\medskip
Let us now consider the images of morphisms under a pseudo-identity.

\begin{lemma}[Reduction Lemma]\label{coro}
Let $A$ be a finite-dimensional algebra, and let $(F, \id): \Kb{A}$ $\ra \Kb{A}$ be a pseudo-identity that fixes all truncated triangles. Then the following holds for a morphism $\f=(f^i): \X\ra \Y$ in $\Kb{A}$.

$(1)$ For $i\in \Z$, $F(\f)=\chi_{i,F}(\f)+\g$ with $\g\in \Zmap_{i+1}(A)(\X, \Y)$. In particular, $F(\f)=\f+\h$ with $\h\in \sum_{j\in \Z}\Zmap_j(A)(\X, \Y)$.

$(2)$ Let $n\geq 0$. Assume that $F$ fixes $\K^{\ell\leq n}(A\proj)$. Then
\begin{enumerate}
\item[{\rm (a)}] $F$ fixes $\K^{\ell\leq n+1}(A\proj)$ if and only if $F$ fixes $\K^{0,n+1}(A\proj)$.
\item[{\rm (b)}] If $\X$ and $\Y$ are in $\K^{0,n+1}(A\proj)$, then there is a morphism $\g\in \bigcap_{1\leq j\leq n+1}\Zmap_j(A)(\X, \Y)$ such that $F(\f)=\f+\g$.
\end{enumerate}
\end{lemma}

\begin{proof}
(1) Note that $\Delta_i(\X)$ is a truncated triangle and therefore a standard triangle for all $i$ and $\X$. Applying the pseudo-identity $F$ to $(*)$ and considering the morphisms $F(\f)$ and $\chi_{i,F}(\f)$, we deduce from $(**)$ and Lemma \ref{st}(4) that the first statement in (1) holds.

To get the second statement in (1), we may assume $\f\ne 0$ in $\Kb{A}$ and choose biggest $m_f\in \Z$ and smallest $n_f\in \Z$ such that $X^j=0=Y^j$ whenever $j$ is outside of $\{m_f, m_f+1,\cdots, n_f\}$. We prove the statement by induction on $n_f-m_f$. If $n_f-m_f=0$, then $F(\f)=\f$ since $F$ is a pseudo-identity. Assume that the statement holds for all morphisms $\g$ with $n_g-m_g<t$. Now assume $n_f-m_f=t$.

For simplicity, we write $m=m_f$ and $n=n_f$. Since ${\f}^{\leq m}$ is a morphism in $A\proj[-m]$ and $F$ is a pseudo-identity, we have $F({\cpx{f}}^{\leq m})={\cpx{f}}^{\leq m}$. By induction hypothesis, $F({\cpx{f}}^{\geq m+1})={\cpx{f}}^{\geq m+1}+\cpx{s}$ for some $\cpx{s}\in \sum_{i\in\mathbb{Z}}\Zmap_i(A)({\X}^{\geq m+1},{\Y}^{\geq m+1})$. Considering $i=m$ in the first statement, we get $F(\f)=\chi_{m,F}(\f)+\h$ with $\h\in \Zmap_{m+1}(A)(\X, \Y)$. By the definition of $\Zmap_i(A)$, $\s: {\X}^{\geq m+1}\ra {\Y}^{\geq m+1}$ can be extended to the morphism $\cpx{\tilde{s}}: \X\ra \Y$ in $\sum_{i\in\mathbb{Z}}\Zmap_i(A)(\X,\Y)$ by adding the zero map $X^m\ra Y^m$. By the definition of $\chi_{m,F}(\f)$, we have $\chi_{m,F}(\f)=\f+\cpx{\tilde{s}}$. So $F(\f)=\f+\cpx{\tilde{s}}+\h$ with $\cpx{\tilde{s}}+\h\in\sum_{i\in\mathbb{Z}}\Zmap_i(A)(\X,\Y)$.

(2) (a) The necessity is clear. Now assume that $F$ fixes $\K^{0,n+1}(A\proj)$. Since the connecting isomorphism is $\id: F[1]\ra [1]F$, we know that $F$ fixes $\K^{0,n+1}(A\proj)[i]$ for all $i\in\Z$. Now suppose that $\f:\X\ra\Y$ is a nonzero morphism in $\K^{\ell\leq n+1}(A\proj)$ such that $\X$ and $\Y$ do not both belong to $\K^{0,n+1}(A\proj)[i]$ for the same $i$. Define $s:=\mbox{min}\{j\in\mathbb{Z}\mid X^j\ne 0\ne Y^j\}$ and $r:=\mbox{max}\{j\in\mathbb{Z}\mid X^j\ne 0\ne Y^j\}$. Then $r-s\leq n$.
Hence $F({\cpx{f}}^{[s,r]})={\cpx{f}}^{[s,r]}$ since ${\f}^{[s, r]}$ lies in $\K^{\ell\leq n}(A\proj)$. Due to ${\f}^{\geq r+1}=0$, we have $F({\f}^{\geq r+1})=0$. By $(1)$, $F({\f}^{\geq s})=\chi_{r,F}({\f}^{\geq s})+\g_{r+1}$ with $\g_{r+1}\in \Zmap_{r+1}(A)({\X}^{\geq s}, {\Y}^{\geq s})=0$. By definition of $\chi_{r,F}({\f}^{\geq s})$, we have $F({\f}^{\geq s})={\f}^{\geq s}$. On the other hand, due to ${\f}^{\leq s-1}=0$, we get $F({\f}^{\leq s-1})=0$. By $(1)$ again, $F(\f)=\chi_{s-1,F}(\f)+\g_s$ with $\g_s\in \Zmap_s(A)(\X, \Y)=0$. By definition, $\chi_{s-1,F}(\f)=\f$. So $F(\f)=\f$.

(b) Let $0\leq j\leq n$. Observe that ${\cpx{f}}^{\leq j}$ and ${\cpx{f}}^{\geq j+1}$ belong to $\K^{\ell\leq n}(A\proj)$. Hence $F$ fixes ${\cpx{f}}^{\leq j}$ and ${\cpx{f}}^{\geq j+1}$, so $\chi_{j,F}(\f)=\f$. Thus (b) now follows from (1).
\end{proof}

With the help of Theorem \ref{t0} and Lemmas \ref{prop-1} and \ref{coro}, we can give short proofs of some known results on Rickard's question.

\begin{exam}\label{ex}
(1) Rickard's question holds for hereditary algebras \cite{MY}.

Indeed, suppose that $A$ is a hereditary algebra. Let $\X, \Y$ be indecomposable complexes in $\K^{0,1}(A\proj)$. Then $\X=(X^0\raf{d_X^0}X^1)$ and $\Y=(Y^0\raf{d_Y^0}Y^1)$, where $d_X^0$ and $d_Y^0$ are injective. It is easy to see that $\Zmap_1(A)(\X, \Y)=0$. By Theorem \ref{t0} and Lemma \ref{coro}(2), Rickard's question holds for $A$.

(2) Let $\Omega$ be the set of triples $(r, n, m)$ of integers such that $1\le r\le n$ and $m\geq 0$. A finite-dimensional algebra $A$ is called a {\it derived-discrete algebra} \cite{BGS} if $A$ is derived equivalent to either a hereditary algebra of Dynkin type or a $k$-algebra $\Lambda(r, n, m)$ with $(r, n, m)\in \Omega$, where $\Lambda(r, n, m)$ is defined by the quiver
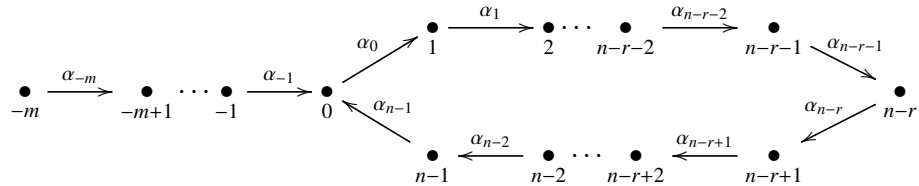
\[\xymatrix@R=0.4em{
&& &\mathop{\bullet}\limits_1\ar[r]^-{\alpha_1} &\mathop{\bullet}\limits_2\cdots \mathop{\bullet}\limits_{n-r-2}\ar[r]^-{\alpha_{n-r-2}} &\mathop{\bullet}\limits_{n-r-1}\ar[rd]^{\alpha_{n-r-1}} &\\
\mathop{\bullet}\limits_{-m}\ar[r]^-{\alpha_{-m}} &\mathop{\bullet}\limits_{-m+1}\cdots  \mathop{\bullet}\limits_{-1}\ar[r]^-{\alpha_{-1}}&\mathop{\bullet}\limits_{0}\ar[ru]^{\alpha_{0}}&&&&\mathop{\bullet}\limits_{n-r}\ar[ld]_-{\alpha_{n-r}}\\
 && &\mathop{\bullet}\limits_{n-1}\ar[lu]_-{\alpha_{n-1}} &\mathop{\bullet}\limits_{n-2}\ar[l]_-{\alpha_{n-2}}\cdots \mathop{\bullet}\limits_{n-r+2}&\mathop{\bullet}\limits_{n-r+1}\ar[l]_-{\alpha_{n-r+1}}&
 }\]
with relations $\alpha_0\alpha_{n-1}=\alpha_{n-1}\alpha_{n-2}= \cdots = \alpha_{n-r+1}\alpha_{n-r}=0$. It was proved, with the help of the description of morphisms of the derived categories of gentle algebras \cite{ALP}, that Rickard's question holds for derived-discrete algebras \cite{BC, CZ}.

For $A:=\Lambda(r, n, m)$ with $r\neq 1$, we give a short proof of this result. Actually, by \cite[Theorem 7.4]{ALP}, $\dim_k\Hom_{\Kb{A}}(\X, \Y)\leq 1$ for all indecomposable objects $\X$ and $\Y$ of $\Kb{A}$. Hence either $\Zmap_i(A)(\X, \Y)=0$ for all $i\in \Z$, or $\Hom_{\Kb{A}}(\X, \Y)=\Zmap_i(A)(\X, \Y)$ for an $i\in \Z$. Thus it follows from Theorem \ref{t0}, Lemma \ref{prop-1} and Lemma \ref{coro}(2) that Rickard's question holds for $A$.
\end{exam}

\section{A series of examples of non-standard derived equivalences}\label{example}

In this section we prove Theorem \ref{t} by constructing infinitely many counterexamples to Rickard's question.

We first prove a criterion for a pair $(\id,\phi)$ to be a triangle functor, where $\phi: \id[1]\ra [1]\id$ is a natural isomorphism.

Let $A$ be an Artin algebra and $\phi: \id [1]\ra [1]\id$ a natural isomorphism on $\Ka$. We say that $\phi$ is {\it represented above} if, for any object $\X$ in $\Ka$ and for any $n\in \mathbb{Z}$, there is a representative $\g: \X\ra \X$ of the homotopy class $\phi_{\X}[-1]: \X\ra \X$ such that
$${\g}^{\geq n}=\phi_{X^{\bullet\geq n}}[-1]: {\X}^{\geq n}\lra {\X}^{\geq n}$$
 in $\Ka$.
\begin{prop}\label{p}
Let $A$ be an Artin algebra. Suppose that $\phi: \id [1]\ra [1]\id$ is a represented above, natural isomorphism on $\Ka$. Then the following are equivalent:

$(1)$ $(\id, \phi): \Ka\ra \Ka$ is a triangle functor.

$(2)$ For any morphism $\f: \X\ra \Y$ in $\Kb{A}$, there exists a morphism $\h: \Y\ra \X$ in $\Kb{A}$ such that $\f\h\f=\f(\phi_{\X}[-1]-1)$ in $\Kb{A}:$
\[\begin{array}{c}\xymatrix{
\X\ar[r]^{\f}\ar[d]_{\phi_{\X}[-1]-1}&\Y\ar@{..>}[dl]_{\h} \\
\X\ar[r]^{\f}&\Y.
}\end{array}\]
\end{prop}
\begin{proof}
Consider the condition $(a)$:

\medskip
$(a)$ For any morphism $\f: \X\ra \Y$ in $\Ka$, there is a morphism $\h: \Y\ra \X$ in $\Ka$ such that $\f\h\f=\f(\phi_{\X}[-1]-1)$ in $\Ka$.

\medskip\noindent
We will prove $(1)\Leftrightarrow (a)\Leftrightarrow (2)$. For simplicity, we write $F$ for the pair $(\id, \phi)$.

$(1)\Rightarrow (a)$ Assume that $F$ is a triangle functor. Let $\f: \X\ra \Y$ be a morphism in $\Ka$.  By the definition of triangulated categories, there is an isomorphism $\cpx{\psi}: \cone(\f)\ra \cone(\f)$ in $\Ka$ such that the following diagram commutes in $\Ka$:
 \[\begin{array}{c}\xymatrix@R=1.5em{
F(\Delta(\f)):& \X\ar[d]^{1}\ar[r]^-{\f}&\Y\ar[r]^-{\theta(\f)}\ar[d]^{1}& \cone(\f)\ar[r]^-{\phi_{\X}\pi(\f)}
\ar@{..>}[d]^-{\cpx{\psi}}&\X[1]\ar[d]^{1} \\
\Delta(\f):& \X\ar[r]^-{\f}&\Y\ar[r]^-{\theta(\f)}& \cone(\f)\ar[r]^-{\pi(\f)}&\X[1].
}\end{array}\]
By Lemma \ref{st}(3), there is a morphism $\h: \Y\ra \X$ in $\Ka$ such that $\f\h\f=\f(\phi_{\X}[-1])-\f=\f(\phi_{\X}[-1]-1)$ in $\Ka$.

$(a)\Rightarrow (1)$ To prove that $F=(\id,\phi)$ is a triangle functor, we need only to show that $F$ preserves triangles. Since every triangle is isomorphic to a standard triangle, it suffices to prove that standard triangles are preserved. Let $\f:\X\to\Y$ be a morphism in $\Ka$, we show that $F(\Delta(\f))$ is a triangle.

Note that $\Ka$ is a Krull-Schmidt category. Thus $\f$ is the direct sum of  an isomorphism and a radical morphism (up to conjugation), that is, there are complexes $\X_1, \X_2, \Y_1, \Y_2$ and  isomorphisms $\cpx{s}_X: \X\ra \X_1\oplus \X_2$ and $\cpx{s}_Y: \Y\ra \Y_1\oplus \Y_2$ such that the following diagram commutes in $\Ka$:
 \[\begin{array}{c}\xymatrix{
\X\ar[d]_{\cpx{s}_X}\ar[r]^{\f}&\Y\ar[d]^{\cpx{s}_Y}\\
\X_1\oplus \X_2\ar[r]^{\left(\begin{smallmatrix}
\cpx{s}&0 \\
0&\cpx{r}
\end{smallmatrix}\right)}&\Y_1\oplus \Y_2,
}\end{array}\]
where $\cpx{s}: \X_1\ra \Y_1$ is an isomorphism and $\cpx{r}: \X_2\ra \Y_2$ is a radical morphism. Then $F(\Delta(\f))$ is a triangle if and only if $F\big(\Delta\left(\begin{smallmatrix}
\cpx{s}&0 \\
0&\cpx{r}
\end{smallmatrix}\right)\big)$ is a triangle if and only if both $F(\Delta(\cpx{s}))$ and $F(\Delta(\cpx{r}))$ are triangles. Clearly, $F(\Delta(\cpx{s}))$ is a triangle as $\cone(\cpx{s})$ is the zero object in $\Ka$. In the following, we show that $F(\Delta(\cpx{r}))$ is a triangle.

Consider the morphism $\cpx{r}$. By the assumption $(a)$, there is a morphism $\h: \Y_2\ra \X_2$ in $\Ka$ such that $\cpx{r}\h\cpx{r}= \cpx{r}(\phi_{\X_2}[-1]-1)$ in $\Ka$.
By Lemma \ref{st}(3)-(5), there is an isomorphism $\cpx{\psi}: \cone(\cpx{r})\ra \cone(\cpx{r})$ in $\Ka$ such that the diagram
 \[\begin{array}{c}\xymatrix@R=1.5em{
F(\Delta(\cpx{r})):& \X_2\ar[d]^{1}\ar[r]^-{\cpx{r}}&\Y_2\ar[r]^-{\theta(\cpx{r})}\ar[d]^{1}& \cone(\cpx{r})\ar[r]^-{\phi_{\X_2}\pi(\cpx{r})}
\ar@{..>}[d]^-{\cpx{\psi}}&\X_2[1]\ar[d]^{1} \\
\Delta(\cpx{r}):& \X_2\ar[r]^-{\cpx{r}}&\Y_2\ar[r]^-{\theta(\cpx{r})}& \cone(\cpx{r})\ar[r]^-{\pi(\cpx{r})}&\X_2[1]
}\end{array}\]
commutes in $\Ka$. This means that $F(\Delta(\cpx{r}))$ is a triangle.

$(a)\Rightarrow (2)$ This implication is clear.

$(2)\Rightarrow (a)$ Let $\f: \X\ra \Y$ be in $\Ka$. Then there is an $n\in \mathbb{Z}$ such that $H^i(\X)=H^i(\Y)=0$ for all $i\leq n$. Recall that $H^i(\X)$ denotes the $i$-th homology of $\X$. Consider the following two complexes
$$\X_0: \;\;\; \cdots \lra 0\lra M\lraf{u} X^n\lraf{d_X^n} X^{n+1}\lraf{d_X^{n+1}}X^{n+2}\lra \cdots$$
$$\Y_0:\;\;\; \cdots \lra 0\lra N\lraf{v} Y^n\lraf{d_Y^n} Y^{n+1}\lraf{d_Y^{n+1}} Y^{n+2}\lra \cdots$$
where $M\raf{u} X^n$ is the kernel of $X^n\raf{d_X^n} X^{n+1}$ and $N\raf{v} Y^n$ is the kernel of $Y^n\raf{d_Y^n} Y^{n+1}$.
Then
$$t_{\X, \Y}: \Hom_{\Ka}(\X, \Y)\lra \Hom_{\K^b(A\md)}(\X_0, \Y_0),\;\;\; \g\longmapsto (H^n({\cpx{g}}^{\geq n}), {\cpx{g}}^{\geq n})$$
is an isomorphism  \cite[Lemma 2.2(1)]{HX10}, where $H^n({\cpx{g}}^{\geq n}): M=H^n({\X}^{\geq n})\ra H^n({\Y}^{\geq n})=N$ is the homology of ${\cpx{g}}^{\geq n}$.
Moreover, if two morphisms $\g_0, \g_1: X^{\geq n}\ra Y^{\geq n}$ of complexes are homotopic, that is, $\g_0\sim \g_1$, then
$$(H^n(\g_0), \g_0)\sim (H^n(\g_1), \g_1): \X_0\lra \Y_0.$$
Consider the morphism ${\cpx{f}}^{\geq n}: {\X}^{\geq n}\ra {\Y}^{\geq n}$ of complexes. By (2), there is a morphism $\h_0: {\Y}^{\geq n}\ra {\X}^{\geq n}$ in $\Kb{A}$ such that ${\cpx{f}}^{\geq n}\h_0 {\cpx{f}}^{\geq n}={\cpx{f}}^{\geq n}(\phi_{X^{\geq n}}[-1]-1) $ in $\Kb{A}$. Then we have a morphism
$\h_1=(H^n(\h_0), \h_0): \Y_0\ra \X_0$
of complexes. Define $\h:=t^{-1}_{\Y, \X}(\h_1)\in \Hom_{\Ka}(\Y, \X).$
Then
 \begin{equation*}
\begin{aligned}
 t_{\X, \Y}(\f\h\f)&=(H^n({\cpx{f}}^{\geq n}\h_0 {\cpx{f}}^{\geq n}), {\cpx{f}}^{\geq n}\h_0 {\cpx{f}}^{\geq n})\\
 &=(H^n({\cpx{f}}^{\geq n}(\phi_{{\X}^{\geq n}}[-1]-1)), {\cpx{f}}^{\geq n}(\phi_{{\X}^{\geq n}}[-1]-1))=t_{\X, \Y}(\f(\phi_{\X}[-1]-1))
\end{aligned}
 \end{equation*}
in $\K^b(A\md)$, where the last equality follows from the condition $\phi$ being represented above. So $\f\h\f=\f(\phi_{\X}[-1]-1)$ in $\Ka$, and therefore $(a)$ holds.
\end{proof}

{\bf In the rest of this section}, let $k:=\mathbb{Z}/2\mathbb{Z}$ and $A$ a finite-dimensional local commutative Frobenius $k$-algebra. Denote by $\rad(A)$ and $\soc(A)$ the radical and socle of $A$, respectively. Thus $\rad(A)\soc(A)=0$ and $\soc(A)\simeq {}_A(A/\rad(A))$ is a simple $A$-module. Assume that $a\in \rad(A)$ and $0\neq s\in \soc(A)\subseteq \rad(A)$. Thus we always have $(\soc(A))^2=0$.

Let $\{a_1, \cdots, a_n\}$ be a $k$-basis of $A$ with $a_n=s$. Let $\omega$ be the following $k$-linear map
$$\omega: A\lra k, \;\;\; \sum_{1\leq i\leq n}k_ia_i\longmapsto k_n, \;k_i\in k \mbox{ for } 1\le i\le n.$$
Then $\omega$ is a Frobenius form of $A$. We also \textbf{assume} that $a$ satisfies $\omega(x^2)=\omega(a x)$ for all $x\in A$.

\begin{lemma}\label{m}  The equality $\omega(\tr(Z^2))=\omega(\tr(a Z))$ holds for all square matrices $Z$ over $A$.
\end{lemma}
\begin{proof}
We calculate $\tr(Z^2)$ as follows:
$$\begin{aligned}
  \tr(Z^2) & =\sum_i(Z^2)_{i, i}=\sum_i\sum_jZ_{i, j}Z_{j, i}
  = \sum_{i<j} Z_{i, j}Z_{j, i} +\sum_{i>j} Z_{i, j}Z_{j, i}+\sum_i Z_{i, i}^2\\
  & = 2\sum_{i<j} Z_{i, j}Z_{j, i} +\sum_i Z_{i, i}^2 = \sum_{i} Z_{i,i}^2.
\end{aligned}$$
Hence
$\omega(\tr(Z^2))=\sum_i \omega(Z_{i, i}^2)=\sum_i \omega(a Z_{i, i})=\omega(\tr(a Z)).$
\end{proof}

\medskip
Now we define a natural transformation $\xi(a): \id[1]\ra [1]\id$ on $\Ka$ as follows.
For a complex $\Y$ in $\Ka$, we define $\xi(a)_{\Y}: \Y[1]\ra \Y[1]$ by
$$\xi(a)_{\Y}^i: Y^{i+1}\lra Y^{i+1},\;\; y\mapsto (1+a)y.$$
Clearly, $\xi(a)_{\Y}: \Y[1]\ra \Y[1]$ is an isomorphism in $\Ka$.

\begin{lemma}\label{k}
For each morphism $\f: \X\ra \Y$ in $\Kb{A}$, there is a morphism $\g: \Y\ra \X$ in $\Kb{A}$ such that $\f\g\f= \f(\xi(a)_{\X}[-1]-1)$ in $\Kb{A}:$
\[\begin{array}{c}\xymatrix{
\X\ar[r]^{\f}\ar[d]_{\xi(a)_{\X}[-1]-1}&\Y\ar@{..>}[dl]_{\g} \\
\X\ar[r]^{\f}&\Y.
}\end{array}\]
\end{lemma}

\begin{proof}
Let $\f: \X\ra \Y$ be a morphism in $\Kb{A}$. We may assume that $\f$ is a representative of the corresponding homotopy class of $\f$. Thus $\f$ is a morphism of complexes. We need to prove the following property:

\medskip
$(P)$ There is a morphism $\g: \Y\ra \X$ of complexes such that $\f\g\f\sim \f(\xi(a)_{\X}[-1]-1)$.

\medskip
\noindent
By definition, we have $\f(\xi(a)_{\X}[-1]-1)=a\f$. So property (P) is equivalent to saying that
$$(\dag)\quad a \f\in \{\f\g\f+d_D^{-1}\h\,|\,\g\in Z^0(\cpx{C}), \h\in D^{-1}\}=\f Z^0(\cpx{C})\f+B^0(\cpx{D})$$
with $\cpx{C}:=\Hom^{\bullet}(\Y, \X)$ and $\cpx{D}:=\Hom^{\bullet}(\X, \Y)$.

To prove $(\dag)$, we shall use the non-degenerate bilinear forms $\lan -, -\ran: C^0\times D^0\ra k$ and $\lan -, -\ran: D^0\times C^0\ra k$ in Lemma \ref{form}. Indeed, for any $\g, \cpx{z}\in C^0$, we have
$$\lan \f\g\f, \cpx{z}\ran =\sum_{i\in \mathbb{Z}}(-1)^i\omega\big(\tr(f^ig^if^iz^i)\big)=\sum_{i\in \mathbb{Z}}(-1)^i\omega\big(\tr(g^if^iz^if^i)\big)=\lan \g, \f\cpx{z}\f\ran.$$
Lemma \ref{form}(3) states that $Z^0(\cpx{C})^{\perp}=B^0(\cpx{D})$. Letting $\cpx{g}$ range over $Z^0(\cpx{C})$, we obtain
$$\f\cpx{z}\f\in B^0(\cpx{D})\iff \cpx{z}\in  (\f Z^0(\cpx{C})\f)^{\perp}.$$
We also have $B^0(\cpx{D})^{\perp}=Z^0(\cpx{C})$ by Lemma \ref{form}(3). Then it follows from
\[\big(\f Z^0(\cpx{C})\f+B^0(\cpx{D})\big)^{\perp}=\big(\f Z^0(\cpx{C})\f\big)^{\perp}\cap \big(B^0(\cpx{D})\big)^{\perp}\]
that $\cpx{w}$ in $C^0$ belongs to $\big(\f Z^0(\cpx{C})\f+B^0(\cpx{D})\big)^{\perp}$ if and only if both $\cpx{w}\in Z^0(\cpx{C})$ and $\f\cpx{w}\f\in B^0(\cpx{D})$ hold true. Now for such a $\cpx{w}$, one has
$$\lan a \f, \cpx{w}\ran =\sum_{i\in \mathbb{Z}}(-1)^i\omega(\tr(a f^iw^i))=\sum_{i\in \mathbb{Z}}(-1)^i\omega(\tr(f^iw^if^iw^i))=\lan \f\cpx{w}\f, \cpx{w}\ran=0,$$
where the second equality follows from Lemma \ref{m}, and the last one follows from Lemma \ref{form}(3). This implies that
$a\f\in \big(\f Z^0(\cpx{C})\f+B^0(\cpx{D})\big)^{\perp\perp}=\f Z^0(\cpx{C})\f+B^0(\cpx{D}).$
Thus $\f$ has property $(P)$, and this finishes the proof.
\end{proof}

\begin{theorem}\label{3.1}
If $a^2\neq 0$, then
 $(\id, \xi(a)): \Ka\ra \Ka$
  is a non-standard derived equivalence.
\end{theorem}
\begin{proof}
For simplicity, we write $\xi$ for $\xi(a)$ and $F:=(\id, \xi): \Ka\ra \Ka$. For $z\in A$, let $\rho(z)$ be the right multiplication map:
 $$\rho(z): A\lra A, x\longmapsto xz \mbox{ for } x\in A.$$

Clearly, $\xi: \id[1]\ra [1]\id$ on $\Ka$ is represented above. By Proposition \ref{p} and Lemma \ref{k}, $F$ is a derived equivalence.

Suppose that $F$ is a standard derived equivalence. By Lemma \ref{idd}, $F$ is naturally isomorphic to  $(\id, \id)$ as triangle functors. Let $\mu: F=(\id,\xi)\ra (\id, \id)$ be a natural isomorphism of triangle functors. We have the commutative diagram in $\Ka$:
  \[\begin{array}{c}\xymatrix@R=1.5em{
F(A[1])=A[1]\ar[r]^{\xi_{A}}\ar[d]_{\mu_{A[1]}} &F(A)[1]=A[1]\ar[d]^{\mu_{A}[1]}\\
\id(A[1])=A[1]\ar[r]^{\id}&\id(A)[1]=A[1]
}\end{array}\]
and $\mu_{A[1]}=\mu_A[1]\circ \xi_{A}$. We may assume $\mu_A:=\rho(u)$ for a unit $u\in A$. Then $\mu_{A[1]}=\rho(u(1+a))[1]$.

Let $C(a):=\cone(\rho(a))$ and $C(s):=\cone(\rho(s))$. Since $\mu$ is a natural isomorphism, we have the commutative diagram in $\Ka$:
\[\xymatrix@R=1.5em{
\Delta(\rho(a)): &A\ar[r]^{\rho(a)}&A\ar[d]^{\mu_{A}}\ar[r]^-{\theta(\rho(a))} &C(a)\ar[d]^{\mu_{C(a)}}\ar[r]^{\pi(\rho(a))}&A[1]\ar[d]^{\mu_{A[1]}} \\
\Delta(\rho(a)): &A\ar[r]^{\rho(a)}&A\ar[r]^-{\theta(\rho(a))} &C(a)\ar[r]^{\pi(\rho(a))}&A[1].}\]
Replacing $\mu_{C(a)}: C(a)\ra C(a)$ with the morphism $(\rho(u), \rho(u)): C(a)\ra C(a)$ of complexes, the diagram is still commutative in $\Ka$. By Lemma \ref{st}(4), $\mu_{C(a)}=(\rho(u), \rho(u+a x))$ for some $x\in A$ with $a xa=0$. Since $A$ is a local algebra, it follows from $a^2\neq 0$ that $x\in \rad(A)$. Similarly, consider the following commutative diagram in $\Ka$:
\[\xymatrix@R=1.5em{
\Delta(\rho(s)): &A\ar[r]^{\rho(s)}&A\ar[d]^{\mu_{A}}\ar[r]^-{\theta(\rho(s))} &C(s)\ar[d]^{\mu_{C(s)}}\ar[r]^{\pi(\rho(s))}&A[1]\ar[d]^{\mu_{A[1]}} \\
\Delta(\rho(s)): &A\ar[r]^{\rho(s)}&A\ar[r]^-{\theta(\rho(s))} &C(s)\ar[r]^{\pi(\rho(s))}&A[1].}\]
Replacing $\mu_{C(s)}: C(s)\ra C(s)$ with the morphism $(\rho(u(1+a)), \rho(u)): C(s)\ra C(s)$ of complexes, the diagram is still commutative in $\Ka$. By Lemma \ref{st}(4), $\mu_{C(s)}=(\rho(u(1+a)), \rho(u+s z))$ with $z\in A$.

Since $a^2\neq 0$ and $Aa\ne 0$, the socle of $Aa$ is $\soc(_AA)$, and therefore there is $y\in A$ such that $y a=s$. Since $a^2\neq 0$ and $(\soc(A))^2=0$, we know $a\notin\soc(A)$ and $y\in \rad(A)$. Consider the morphism
 $$\phi:=(\rho(1), \rho(y)): C(a)\lra C(s)$$
in $\Ka$. As $\mu$ is a natural isomorphism, we have the commutative diagram in $\Ka$:
  \[\xymatrix@R=1.5em{
 (\ddag)\quad & C(a)\ar[rr]^{\phi}\ar[d]_{\mu_{C(a)}}&&C(s)\ar[d]^{\mu_{C(s)}}\\
 & C(a)\ar[rr]^{\phi} &&C(s).}\]
We shall prove that this diagram is actually not commutative, and hence a contradiction appears. Indeed, since both $y$ and $x$ are in $\rad(A)$, we get $s z y=z s y=0$ and $y a x=s x=0$. Then
$$\mu_{C(s)}\phi-\phi\mu_{C(a)}=(\rho(u(1+a)), \rho(u y+s z y))-(\rho(u), \rho(y u+y a x))=(\rho(ua), 0)$$ in $\Ka$.
Suppose that the morphism $$(\rho(u a), 0): C(a)\lra C(s)$$of complexes
is zero-homotopic. Then there is a morphism $\rho(v): A\ra A$ with $v\in A$ such that $\rho(s)\rho(v)=0$ and $\rho(v)\rho(a)=\rho(ua)$. In other words, $sv=0$ and $(u-v)a=0$ in $A$. Since $A$ is a local algebra, every element in $A\setminus\rad(A)$ is a unit. It follows from $sv=0$ and $s\ne 0$ that $v\in \rad(A)$. As $u$ is a unit in $A$, the element $u-v$ is a unit in $A$, and therefore $a=0$. This contradicts the assumption $a^2\neq 0$. Consequently, the diagram $(\ddag)$ is not commutative in $\Ka$. This contradiction shows that $F=(\id, \xi)$ is not standard.
\end{proof}

{\it Remark.}  (1) The restricted triangle equivalence $(\id, \xi(a)): \Kb{A}\ra \Kb{A}$ is also non-standard. Indeed, the proof of Theorem \ref{3.1} shows that the obstruction to a standard equivalence already appears in $\Kb{A}$.

(2) Let us explain why we take $k=\mathbb{Z}/2\mathbb{Z}$. The key point is the  condition that  $\omega(x^2)=\omega(ax)$ holds for all $x\in A$. This happens only if $k=\mathbb{Z}/2\mathbb{Z}$. Actually, if we choose $x\in A$ such that $\omega(x^2)=\omega(ax)\neq 0$, then, for each $\lambda\in k$, we get $\lambda^2\omega(x^2)=\omega((\lambda x)^2)=\omega(a\lambda x)=\lambda\omega(ax)$. It follows that $\lambda^2=\lambda$ for all $\lambda\in k$. This can happen only in $\mathbb{Z}/2\mathbb{Z}$.

(3) Since we take $k=\mathbb{Z}/2\mathbb{Z}$, both sides of  the equality $\omega(x^2)=\omega(a x)$ are $k$-linear in $x$. To verify whether $\omega(x^2)=\omega(a x)$ for all $x\in A$, it suffices to check the equality for $x$ running over a $k$-basis of $A$.

\medskip
Let $A, B$ be  finite-dimensional, non-simple, local commutative Frobenius $k$-algebras with the Frobenius forms $\omega_A$ and $\omega_B$, respectively. Assume that $k$ is a splitting field for both $A$ and $B$, that is, $A/\rad(A)\simeq k\simeq B/\rad(B)$ as $k$-algebras.  Then $\soc(A)\subseteq \rad(A)$, $\soc(B)\subseteq \rad(B)$, and the tensor product algebra $A\otimes_k B$ is still a finite-dimensional local commutative Frobenius $k$-algebra with the Frobenius form defined by
$$\omega_A\otimes \omega_B: A\otimes_k B\lra k,\;\; \sum_ix_i\otimes y_i\longmapsto \sum_i\omega_A(x_i)\omega_B(y_i).$$
Since $k$ is a splitting field for $A$ and $B$, we have $\soc(A\otimes_kB)=\soc(A)\otimes_k\soc(B)$ and $\rad(A\otimes_kB)=\rad(A)\otimes_kB+A\otimes_k\rad(B)$.

\begin{prop} \label{prop4.6} Let $A,B$ be as above, and let $a\in \rad(A)$ and $b\in \rad(B)$ such that $a^2\neq 0$ and $\omega_A(x^2)=\omega_A(ax)$ for all $x\in A$; $b^2\neq 0$ and $\omega_B(y^2)=\omega_B(b y)$ for all $y\in B$. Let $\Lambda:= A\otimes_k B$ be the tensor product of $A$ and $B$ over $k$.
Then
$$(\id, \xi(a\otimes b)): \KL\lra \KL$$
 is a non-standard derived equivalence.
\end{prop}
\begin{proof}
Clearly, $(a\otimes b)^2\neq 0$. Consider elements of $A\otimes_k B$ in the form $x\otimes y$ with $x\in A$ and $y\in B$. For $x\otimes y\in A\otimes_kB$, we have
$$\begin{aligned}
  (\omega_A\otimes \omega_B)((x\otimes y)^2) & =(\omega_A\otimes \omega_B)(x^2\otimes y^2)=\omega_A(x^2)\omega_B(y^2)\\
  & =\omega_A(a x)\omega_B(b y)=(\omega_A\otimes \omega_B)((a\otimes b)(x\otimes y)).
\end{aligned}$$
Since $k=\Z/2\Z$, the map $z\mapsto z^2$ is $k$-linear. It follows that $(\omega_A\otimes \omega_B)(z^2)= (\omega_A\otimes \omega_B)((a\otimes b)z)$ for all $z\in A\otimes_kB$.
Thus Proposition \ref{prop4.6} follows from Theorem \ref{3.1} immediately.
\end{proof}

Roughly speaking, Proposition \ref{prop4.6} shows that once we know a counterexample to Rickard's question, we can construct infinitely many counterexamples by tensor products of algebras.

\medskip
The following is a concrete example satisfying conditions in Theorem \ref{3.1}. Due to $\Db{A}\simeq \Ka$ as triangulated categories,
Theorem \ref{t} becomes a consequence of Theorem \ref{3.1} and the following example.

\begin{exam}\label{counterexa}
Let $m\geq 1$ and $n_i\geq 1$ for all $1\leq i\leq m$.  Let
$$A:=k[x_1,x_2,\ldots,x_m]/(x_1^{2n_1+1},x_2^{2n_2+1},\ldots,x_m^{2n_m+1}).$$
Then $A$ is the tensor product of $A_i:=k[x_i]/(x_i^{2n_i+1})$ for $1\le i\le m$. For each $i$, let $a_i=x_i^{n_i}$ and let $\omega_i:A_i\ra k$ be the linear map extracting the coefficient of $x_i^{2n_i}$. Then, for each $x_i^t$, it is easy to check that
$\omega_i((x_i^t)^2)=\omega_i(x_i^{n_i+t})=\omega_i(x_i^{n_i}x_i^t)$
and $a_i^2=x_i^{2n_i}\neq 0$ in $A_i$. Taking tensor products, we deduce by Proposition \ref{prop4.6} that $(\id,\xi(a_1a_2\cdots a_m)):\Ka\ra\Ka$ is a non-standard derived equivalence.
\end{exam}

Remark that after this paper appeared in arXiv, Dr. Jinbi Zhang from Anhui University modified Example \ref{counterexa} and gave
counterexamples, to Rickard's question, of algebras over fields of other characteristic.

 \smallskip
{\bf Acknowledgements:} The research work was supported partially by the National Natural Science Foundation of China (Grants 12671048, 12501044) and Beijing Natural Science Foundation (No. 1252011). The third named author is supported by the China Postdoctoral Science Foundation (Grant 2025M783067). During a visit of CCXi to Southern University of Science and Technology in August 2026, a preliminary version of the work was completed. Xi is very grateful to Professor Zhicheng Feng for his invitation and warm hospitality. The author Jin Zhang would like to thank the School of Mathematical Sciences of Capital Normal University for friendly hospitality.

When considering the proof of Lemma \ref{k} in the case $A=k[x]/(x^3)$, we asked ChatGPT Pro for help. It suggested a similar condition.


\medskip
\begin{thebibliography}{abcd}
\bibitem{ALP} K. K. Arnesen, R. Laking, D. Pauksztello, Morphisms between indecomposable complexes in the bounded derived category of a gentle algebra. J. Algebra {\bf 467} (2016) 1-46.

\bibitem{BGS} G. Bobi\'{n}ski, C. Gei\ss, A. Skowro\'{n}ski, Classification of discrete derived categories. Cent. Eur. J. Math. {\bf 2} (2004) 19-49.

\bibitem{BC} G. Bobi\'nski, T. Ciborski, Derived equivalences for the derived discrete algebras are standard. Algebr. Represent. Theory {\bf 29} (2026) 331-352.

\bibitem{C} X. W. Chen, A note on standard equivalences. Bull. Lond. Math. Soc. {\bf 48} (2016) 797-801.

\bibitem{C1} X. W. Chen, Representability and autoequivalence groups. Math. Proc. Cambridge Philos. Soc. {\bf 171} (2021) 657-668.

\bibitem{CY} X. W. Chen, Y. Ye, The {$\mathbf{D}$}-standard and {$\mathbf{K}$}-standard categories. Adv. Math. {\bf 333} (2018) 159-193.

\bibitem{CZ} X. W. Chen, C. Zhang, The derived-discrete algebras and standard equivalences. J. Algebra {\bf 525} (2019) 259-283.

\bibitem{H} Happel, D.: \emph{Triangulated categories in the representation theory of finite-dimensional algebras}. London Mathematical Society Lecture Note Series, 119. Cambridge University Press, Cambridge, 1988.

\bibitem{HX10} W. Hu, C. C. Xi, Derived equivalences and stable equivalences of Morita type. I. Nagoya Math. J. {\bf 200} (2010) 107-152.


\bibitem{Krause} H. Krause, Krull--Schmidt categories and projective covers,
Expos. Math. \textbf{33} (2015), no.~4, 535-549.

\bibitem{MY} J. I. Miyachi, A. Yekutieli, Derived Picard groups of finite-dimensional hereditary algebras. Compositio Math. {\bf 129} (2001) 341-368.

\bibitem{N} A. Neeman, Some new axioms for triangulated categories. J. Algebra {\bf 139} (1991) 221-255.

\bibitem{R0} J. Rickard, Morita theory for derived categories. J. Lond. Math. Soc. {\bf 39} (1989) 436-456.

\bibitem{R} J. Rickard, Derived equivalences as derived functors. J. Lond. Math. Soc. {\bf 43} (1991) 37-48.

\bibitem{RZ} R. Rouquier, A. Zimmermann, Picard groups for derived module categories. Proc. Lond. Math. Soc. {\bf 87} (2003) 197-225.

\bibitem{Y}  A. Yekutieli, Dualizing complexes, Morita equivalences and the derived Picard group of a ring. J. London Math. Soc. {\bf 60} (1999) 723-746.
\end{thebibliography}
\end{document}